\documentclass[11pt]{amsart}

\usepackage[utf8]{inputenc}
\usepackage[T1]{fontenc}
\usepackage[english]{babel}
\usepackage{hyperref}

\usepackage{amsfonts}
\usepackage{amsmath}
\usepackage{amssymb}
\usepackage{amstext}
\usepackage{amsthm}
\usepackage{amsxtra}
\usepackage{epic, eepic}
\usepackage{epsfig,color}
\usepackage[usenames,dvipsnames]{xcolor}
\usepackage{fullpage}
\usepackage{shuffle}
\usepackage{vmargin}
\usepackage{MnSymbol}

\usepackage{tikz}
\usetikzlibrary{backgrounds,shapes}

\theoremstyle{plain}

\theoremstyle{definition}
\newtheorem{theorem}{Theorem}[section]

\newtheorem{corollary}[theorem]{Corollary}

\newtheorem{lemma}[theorem]{Lemma}
\newtheorem{proposition}[theorem]{Proposition}
\newtheorem*{maintheorem}{Theorem~\ref{thm:mainidentity}}
\newtheorem*{mainconseq1}{Theorem~\ref{thm:mainconseq1}}
\newtheorem*{mainconseq2}{Theorem~\ref{thm:mainconseq2}}
\newtheorem*{mygeneralsummation}{Theorem~\ref{thm:mygeneralsummation}}

\theoremstyle{remark}

\newtheorem{remark}[theorem]{Remark}

\newcommand{\qbinom}[2]{\genfrac{[}{]}{0pt}{}{#1}{#2}}

\newcommand{\<}{\langle}
\renewcommand{\>}{\rangle}

\newcommand{\NN}{\mathbb{N}} 
\newcommand{\QQ}{\mathbb{Q}} 
\newcommand{\wH}{\widetilde{H}} 
\newcommand{\T}{\mathcal{T}} 
\newcommand{\Pc}{\mathcal{P}} 
\newcommand{\tTheta}{\tilde{\Theta}} 
\newcommand{\Exp}{\mathrm{Exp}} 
\newcommand{\PPi}{\mathbf{\Pi}\mkern 1mu } 

\numberwithin{equation}{section}

\title{New identities for Theta operators}

\author[M. D'Adderio]{Michele D'Adderio}
\address{Universit\'e Libre de Bruxelles (ULB), D\'epartement de Math\'ematique, Boulevard du Triomphe, B-1050 Bruxelles, Belgium}\email{mdadderi@ulb.ac.be}

\author[M. Romero]{Marino Romero}
\address{University of California, San Diego, 
	Department of Mathematics,  
	9500 Gilman Drive \# 0112,
	La Jolla, CA  92093-0112, USA}\email{mar007@ucsd.edu}

\begin{document}

\maketitle

\begin{abstract}
	In this article, we prove a new general identity involving the Theta operators introduced by the first author and his collaborators in \cite{dadderio2019theta}. From this result, we can easily deduce several new identities that have combinatorial consequences in the study of Macdonald polynomials and diagonal coinvariants. In particular, we provide a unifying framework from which we recover many identities scattered in the literature, often resulting in drastically shorter proofs.
\end{abstract}

\tableofcontents

\section*{Introduction}

In 1988 Macdonald introduced a fundamental basis of symmetric functions depending on two parameters, and he conjectured certain positivities for its elements \cite{Macdonald-Book-1995}. In their effort to solve these conjectures, Garsia and Haiman introduced in the nineties a modified version of this basis \cite{Garsia-Haiman-PNAS-1993} and initiated the study of diagonal coinvariants of the symmetric group \cite{Garsia-Haiman-qLagrange-1996}. This led to the introduction of the famous nabla operator ($\nabla$) and its siblings, the Delta operators ($\Delta_f$ and $\Delta'_f$) \cite{Bergeron-Garsia-Haiman-Tesler-Positivity-1999}, which are diagonal operators with respect to the modified Macdonald basis. 

In the last twenty years, several combinatorial formulas have been conjectured for the applications of some of these operators to basic symmetric functions. The most famous of these stories is certainly the one of the \emph{shuffle conjecture} for $\nabla e_n$ \cite{HHLRU-2005}, which gives a combinatorial formula for the Frobenius characteristic of the coinvariants of the diagonal action of $\mathfrak{S}_n$ on $\mathbb{C}[\mathbf{x},\mathbf{y}]=\mathbb{C}[x_1,x_2,\dots,x_n,y_1,y_2,\dots,y_n]$ \cite{Haiman-Vanishing-2002}. This long standing conjecture has recently been proved by Carlsson and Mellit \cite{Carlsson-Mellit-ShuffleConj-2015}, who actually proved a compositional refinement \cite{Haglund-Morse-Zabrocki-2012} of the conjecture. 

After the announcement of this result, a lot of attention has been dedicated to the so-called Delta conjecture for $\Delta_{e_{n-k-1}}' e_n$ \cite{Haglund-Remmel-Wilson-2015}, which is a natural extension of the shuffle conjecture that gives, conjecturally, the Frobenius characteristic of the coinvariants for the diagonal action of $\mathfrak{S}_n$ on the exterior algebra $\mathbb{C}[\mathbf{x},\mathbf{y}]\<\mathbf{\theta}\>=\mathbb{C}[\mathbf{x},\mathbf{y}]\<\theta_1,\dots,\theta_n\>$ on $n$ free generators with coefficients in $\mathbb{C}[\mathbf{x},\mathbf{y}]$ \cite{Zabrocki_Delta_Module}. 

In an effort to prove the Delta conjecture, in \cite{dadderio2019theta} the authors introduced the so-called Theta operators and showed how these operators give, on one hand, a conjectural formula for the coinvariants of the diagonal action of $\mathfrak{S}_n$ on the exterior algebra on $n$ free generators with coefficients in $\mathbb{C}[\mathbf{x},\mathbf{y}]\<\mathbf{\theta}\>$, and on the other hand, a compositional refinement of the Delta conjecture (extending the one in \cite{Haglund-Morse-Zabrocki-2012}), which has been recently proved in \cite{DAdderio_Mellit2020}.

In the present article we bring further support for the centrality of the Theta operators by proving a new general identity. From this new identity, we will prove a plethora of results: we will get several new identities, and we will show how these new identities form a natural framework from which one can prove identities appearing in the literature. Besides giving a unifying view of all these results, our new proofs will often be substantially shorter than the original ones. Many of these identities have been used to prove combinatorial consequences of the Delta conjecture and some related more recent conjectures \cite{DAdderio-Iraci-VandenWyngaerd-Delta-Square,DAdderio-Iraci-VandenWyngaerd-Bible}. It is important to note that behind many of these identities, there often lies combinatorial identities with remarkable consequences.

\smallskip

The main new general identity of this article is stated in the following theorem. The main tool for the proof is the five-term relation in \cite{Garsia_Mellit}.
\begin{maintheorem}
We have
\[\tTheta(z,v)^{-1} \T_u \tTheta(z,v) \T_{u}^{-1} = \Exp\left[\frac{uz(v-1)}{M}\right]\Delta_{uzv} ,\]
where $\Exp$ denotes the plethystic exponential, $\Delta_v\coloneq \sum_{n\geq 0}(-v)^n\Delta_{e_n}$ is the generating function of Delta operators, $\tTheta(z,v)=\Delta_v \Pc_{-\frac{z}{M}} \Delta_v^{-1}$ is the generating function of Theta operators, and $\T_u=\sum_{n\geq 0}u^nh_n^\perp$ is the usual plethystic translation operator.
\end{maintheorem}
Two of the main consequences of this identity are the following theorems. The first one follows almost immediately.
\begin{mainconseq1}
	We have
	\[h_j^\perp\Theta_{e_k} =\sum_{r=0}^j\Theta_{e_{k-j+r}}\Delta_{e_{j-r}}h_r^{\perp},\quad e_j^\perp \Theta_{e_k} =\sum_{r=0}^j \Theta_{e_{k-j+r}} e_r^\perp \Delta_{h_{j-r}}\quad\text{and}\quad 	h_j^\perp \Theta_{h_k} =\sum_{r=0}^j\Delta_{h_{j-r}} \Theta_{h_{k-j+r}}h_r^\perp. \]
\end{mainconseq1}
The second consequence of our theorem will also use Tesler's identity, Theorem~\ref{thm:Teslersidentity}.
\begin{mainconseq2}
	For every partition $\mu\vdash k$ and every $F\in\Lambda^{(n)}$ we have
	\[ \<h_{k}^\perp\Theta_{e_{n}} \wH_\mu , F\>_* =F[MB_\mu], \]
	where $\wH_\mu$ is the modified Macdonald polynomial indexed by $\mu$.
\end{mainconseq2}
From these two theorems and some basic standard tools of symmetric function theory we are able to derive plenty of identities known in the literature, often in a much simpler way, and sometimes we extend them substantially. For example we can prove quite easily the following identity, which is a strong extension of \cite[Theorem~4.6]{DAdderio-Iraci-VandenWyngaerd-Bible}, whose original proof alone took about 15 pages, and which might provide some insight in the so-called Theta conjecture in \cite[Section~9]{dadderio2019theta}.

\begin{mygeneralsummation}
	Given $j,m,\ell,k\in\NN$, $k\geq 1$ we have
	\begin{align*} 
	 & h_{j}^\perp\Theta_{e_m}\Theta_{e_{\ell}}\widetilde{H}_{(k)}=\\
	& = \sum_{r=0}^{j}\qbinom{k}{r}_q\sum_{a=0}^{k}\sum_{b=1}^{j-r+a} q^{\binom{k-r-a}{2}}\qbinom{b-1}{a}_q  \qbinom{b+r-a-1}{k-a-1}_q \Theta_{e_{m-j+r}}\Theta_{e_{\ell+k-j-a}}\Delta_{e_{j-r+a}}  E_{j-r+a,b} \\
	 & + \sum_{r=0}^{j}\qbinom{k}{r}_q\sum_{a=0}^{k}\sum_{b=1}^{j-r+a} q^{\binom{k-r-a+1}{2}}\qbinom{b-1}{a-1}_q \qbinom{b+r-a}{k-a}_q  \Theta_{e_{m-j+r}}\Theta_{e_{\ell+k-j-a}}\Delta_{e_{j-r+a}} E_{j-r+a,b} .
	\end{align*}
\end{mygeneralsummation}

The rest of the paper is organized in the following way. In Section~1 we collect the definitions and results needed in the rest of the article. In Section~2 we prove our main result. In Sections~3 and 4 we give the main consequences of our general identity. In Sections~5 and 6 we develop several consequences of Sections~3 and 4 respectively. In Section~7 we provide a new simplified treatment of several known results in the literature, and in Sections~8 and 9 we deduce further new identities from our development.

\section*{Acknowledgments}
The authors are grateful to Alessandro Iraci and Anna Vanden Wyngaerd for interesting discussions, in particular for suggesting the identity in Corollary~\ref{cor:genDeltaId}, and to Anton Mellit for encouraging this collaboration.

The first author is partially supported by the Fonds Thelam project J1150080.

The second author was partially supported by the University of California President’s Postdoctoral Fellowship.

\section{Background material}

\subsection{$q$-notation and elementary identities}

For $n, k\in \mathbb{N}$, we set
\begin{align}
[0]_q \coloneq 0, \quad \text{ and } \quad [n]_q \coloneq \frac{1-q^n}{1-q} = 1+q+q^2+\cdots+q^{n-1} \quad \text{ for } n \geq 1;
\end{align}
\begin{align}
[0]_q! \coloneq 1 \quad \text{ and }\quad [n]_q! \coloneq [n]_q[n-1]_q \cdots [2]_q [1]_q \quad \text{ for } n \geq 1; \text{ and }
\end{align}
\begin{align}
\qbinom{n}{k}_q  \coloneq \frac{[n]_q!}{[k]_q![n-k]_q!} \quad \text{ for } n \geq k \geq 0, \quad \text{ and } \quad \qbinom{n}{k}_q \coloneq 0 \quad \text{ for } n < k.
\end{align}
Recall the well-known recurrence
\begin{equation}\label{eq:qrecurrence}
\qbinom{n}{k}_q=q^k\qbinom{n-1}{k}_q+\qbinom{n-1}{k-1}_q = \qbinom{n-1}{k}_q+q^{n-k}\qbinom{n-1}{k-1}_q.
\end{equation}

Recall also the standard notation for the $q$-\emph{rising factorial}  
\begin{align}
(a;q)_0 & \coloneq 1; & &&
(a;q)_s & \coloneq (1 - a)(1 - qa)(1 - q^2 a) \cdots (1 - q^{s-1} a)\text{ for }s\geq 1.
\end{align}
It is well known (cf. \cite[Theorem~7.21.2]{Stanley-Book-1999}) that (recall that $h_0 = 1$)
\begin{align}
\label{eq:h_q_binomial}
h_k[[n]_q] = \frac{(q^{n};q)_k}{(q;q)_k} = \qbinom{n+k-1}{k}_q \quad \text{ for } n \geq 1 \text{ and } k \geq 0,
\end{align}
and (recall that $e_0 = 1$)
\begin{align} \label{eq:e_q_binomial}
e_k[[n]_q] = q^{\binom{k}{2}} \qbinom{n}{k}_q \quad \text{ for all } n, k \geq 0.
\end{align}
Also (cf. \cite[Corollary~7.21.3]{Stanley-Book-1999})
\begin{align}
\label{eq:h_q_prspec}
h_k\left[\frac{1}{1-q}\right] = \frac{1}{(q;q)_k} = \prod_{i=1}^k \frac{1}{1-q^i} \quad \text{ for } k \geq 0,
\end{align}
and
\begin{align}
\label{eq:e_q_prspec}
e_k\left[\frac{1}{1-q}\right] = \frac{q^{\binom{k}{2}}}{(q;q)_k} = q^{\binom{k}{2}} \prod_{i=1}^k \frac{1}{1-q^i} \quad \text{ for } k \geq 0.
\end{align}
The following two identities are classical (see (3.3.6) and (3.3.10) in \cite[Section~3.3]{Andrews-Book-Partitions} respectively):
\begin{equation} \label{eq:qbinom}
\sum_{j=0}^n(-x)^jq^{\binom{j}{2}}\qbinom{n}{j}_q=(x;q)_n\quad  \text{($q$-binomial theorem)}
\end{equation}
and
\begin{equation} \label{eq:qChuVander}
\sum_{j=0}^k q^{(n-j)(k-j)}\qbinom{n}{j}_q\qbinom{m}{k-j}_q = \qbinom{m+n}{k}_q\text{ ($q$-Chu-Vandermonde)}.
\end{equation}

We collect here a few elementary lemmas that will be used in the text.

The following lemma is proved in \cite[Lemma~4.11]{dadderio2019theta} (cf.\ also Remark~\ref{rem:elLemma}).
\begin{lemma} \label{lem:elementary1}
	For $i\geq 1$, $a\geq -i$ and $s\geq 0$ we have
	\begin{equation} \label{eq:first_qlemma}
	\sum_{r=1}^{i}\qbinom{i-1}{r-1}_q \qbinom{r+s+a-1}{s-1}_q  q^{\binom{r}{2}+r-ir}(-1)^{i-r} =q^{\binom{i}{2}+(i-1)a} \qbinom{s+a}{i+a}_q.
	\end{equation}
\end{lemma}
The proof of the following three lemmas are in the appendix.
\begin{lemma}
	Given $i,a,b\in\NN$, $b\geq 1$, $i,b\geq a$ we have
	\begin{align} \label{eq:lemmaelem}
	\sum_{c=a}^{i} \qbinom{i-a}{i-c}_q\qbinom{c-a+b-1}{c-1}_q q^{\binom{c}{2}+c-ic}(-1)^{i-c}  = \qbinom{b-1}{i-1}_qq^{\binom{i}{2}-a(i-1)}.
	\end{align}
\end{lemma}

\begin{lemma}
	Given $r,k,a,b\in\NN$ we have
	\begin{align}
	\notag & \sum_{s=0}^rq^{\binom{r-s}{2}}\qbinom{r}{r-s}_q   q^{\binom{a}{2}}\qbinom{k-s}{a}_qq^{\binom{k-s}{2}-a(k-s-1)}\qbinom{b-1}{k-s-1}_q =\\ 
	\label{eq:lemelem2} & =q^{\binom{k-r-a}{2}}\qbinom{b-1}{a}_q  \qbinom{b+r-a-1}{k-a-1}_q+q^{\binom{k-r-a+1}{2}}\qbinom{b-1}{a-1}_q \qbinom{b+r-a}{k-a}_q.
	\end{align}
\end{lemma}
\begin{lemma}
	Given $k,a,b\in\NN$ we have
	\begin{align} 	\label{eq:lemelem3}
	q^{\binom{k-a}{2}}\qbinom{b-1}{a}_q  \qbinom{b-a-1}{k-a-1}_q + q^{\binom{k-a+1}{2}}\qbinom{b-1}{a-1}_q \qbinom{b-a}{k-a}_q & = q^{\binom{k-a}{2}} \qbinom{k}{a}_q  \qbinom{b-1}{k-1}_q.
	\end{align}
\end{lemma}

\subsection{Symmetric function basics}

The main references that we will use for symmetric functions
are \cite{Macdonald-Book-1995}, \cite{Stanley-Book-1999} and \cite{Haglund-Book-2008}. In particular, we will mainly use the notation from \cite{DAdderio-Iraci-VandenWyngaerd-Bible} and \cite{dadderio2019theta}. We just recall here a few definitions and basic results to avoid possible confusion.

\medskip

The standard bases for symmetric functions that will appear in our
calculations are the complete $\{h_{\lambda}\}_{\lambda}$, elementary $\{e_{\lambda}\}_{\lambda}$, power $\{p_{\lambda}\}_{\lambda}$ and Schur $\{s_{\lambda}\}_{\lambda}$ bases.

\medskip

\emph{We will use the usual convention that $e_0=h_0=1$ and $e_k=h_k=0$ for $k<0$.}

\medskip

The ring $\Lambda$ of symmetric functions can be thought of as the polynomial ring in the power
sum generators $p_1, p_2, p_3,\dots$. This ring has a grading $\Lambda=\bigoplus_{n\geq 0}\Lambda^{(n)}$ given by assigning degree $i$ to $p_i$ for all $i\geq 1$. As we are working with Macdonald symmetric functions
involving two parameters $q$ and $t$, we will consider this polynomial ring over the field $\QQ(q,t)$.
We will make extensive use of \emph{plethystic notation}.

Notice that in the plethystic notation, $p_k[-X]$ equals $-p_k[X]$ and not $(-1)^kp_k[X]$. We refer to this as the \emph{plethystic minus sign}. As the latter sort of negative sign can be also useful, it is customary to use the notation $\epsilon$ to express it: we will have $p_k[\epsilon X] = (-1)^k p_k[X]$, so that, in general,
\begin{align}
\label{eq:minusepsilon}
f[-\epsilon X] = \omega f[X]
\end{align}
for any symmetric function $f$, where $\omega$ is the fundamental algebraic involution which sends $e_k$ to $h_k$, $s_{\lambda}$ to $s_{\lambda'}$ and $p_k$ to $(-1)^{k-1}p_k$.

\medskip

We denote by $\<\, , \>$ the \emph{Hall scalar product} on symmetric functions, which can be defined by saying that the Schur functions form an orthonormal basis. It is given by setting
\[ \< p_{\lambda},p_{\mu} \> = z_{\mu}\chi(\lambda=\mu), \]
where $\chi$ is the indicator function which is defined as $\chi(\mathcal{P})=1$ if $\mathcal{P}$ is true and $\chi(\mathcal{P})=0$ otherwise, and $z_\mu = 1^{m_1}m_1! ~2^{m_2} m_2! \cdots$, where $m_i$ is the multiplicity of $i$ in $\mu$.

With the symbol ``$\perp$'' we denote the operation of taking the adjoint of an operator with respect to the Hall scalar product, i.e.
\begin{equation}
\langle f^\perp g,h\rangle=\langle g,fh\rangle\quad \text{ for all }f,g,h\in \Lambda.
\end{equation}

Recall the \emph{Cauchy identities}
\begin{align}
	\label{eq:Cauchy_identities}
	e_n[XY] = \sum_{\lambda\vdash n} s_{\lambda}[X] s_{\lambda'}[Y] \quad \text{ and } \quad h_n[XY] = \sum_{\lambda\vdash n} s_{\lambda}[X] s_{\lambda}[Y].
\end{align}

The following specialization is an easy consequence of the well-known identity
\begin{equation} \label{eq:SchurSummation}
s_\lambda[X+Y]=\sum_{\mu\subseteq \lambda}s_\mu[X]s_{\lambda/\mu}[Y]. 
\end{equation}
For $\lambda\vdash n$ we have
\begin{equation} \label{eq:slambda1mv}
s_\lambda[1-v]=\left\{\begin{array}{ll}
(-v)^k(1-v) & \text{if }\lambda=(n-k,1^k)\text{ for some }k\in\{0,1,\dots,n-1\}\\
0 & \text{otherwise.}
\end{array}\right.
\end{equation}
\medskip

If we identify the partition $\mu$ with its Ferrers diagram, i.e. with the collection of cells $\{(i,j)\mid 1\leq i\leq \mu_j, 1\leq j\leq \ell(\mu)\}$, then for each cell $c\in \mu$ we refer to the \emph{arm}, \emph{leg}, \emph{co-arm} and \emph{co-leg} (denoted respectively as $a_\mu(c), l_\mu(c), a'_\mu(c), l'_\mu(c)$) as the number of cells in $\mu$ that are strictly to the right, above, to the left and below $c$ in $\mu$, respectively (see Figure~\ref{fig:notation}).

\begin{figure}[h]
	\centering
	\begin{tikzpicture}[scale=.4]
	\draw[gray,opacity=.4](0,0) grid (15,10);
	\fill[white] (1,10)|-(3,9)|- (5,7)|-(9,5)|-(13,2)--(15.2,2)|-(1,10.2);
	\draw[gray]  (1,10)|-(3,9)|- (5,7)|-(9,5)|-(13,2)--(15,2)--(15,0)-|(0,10)--(1,10);
	\fill[blue, opacity=.2] (0,3) rectangle (9,4) (4,0) rectangle (5,7); 
	\fill[blue, opacity=.5] (4,3) rectangle (5,4);
	\draw (7,4.5) node {\tiny{Arm}} (3.25,5.5) node {\tiny{Leg}} (6.25, 1.5) node {\tiny{Co-leg}} (2,2.5) node {\tiny{Co-arm}} ;
	\end{tikzpicture}
	\caption{}
	\label{fig:notation}
\end{figure}

We set
\begin{equation}
M \coloneq (1-q)(1-t),
\end{equation}
and we define for every partition $\mu$
\begin{align}
	B_{\mu} &  \coloneq B_{\mu}(q,t)=\sum_{c\in \mu}q^{a_{\mu}'(c)}t^{l_{\mu}'(c)}\\
	D_{\mu} &  \coloneq MB_{\mu}(q,t)-1\\
	T_{\mu} &  \coloneq T_{\mu}(q,t)=\prod_{c\in \mu}q^{a_{\mu}'(c)}t^{l_{\mu}'(c)}\\
	\Pi_{\mu} &  \coloneq \Pi_{\mu}(q,t)=\prod_{c\in \mu/(1)}(1-q^{a_{\mu}'(c)}t^{l_{\mu}'(c)})\\
	w_{\mu} &  \coloneq w_{\mu}(q,t)=\prod_{c\in \mu}(q^{a_{\mu}(c)}-t^{l_{\mu}(c)+1})(t^{l_{\mu}(c)}-q^{a_{\mu}(c)+1}).
\end{align}

Notice that
\begin{equation} \label{eq:Bmu_Tmu}
B_{\mu}=e_1[B_{\mu}]\quad \text{ and } \quad T_{\mu}=e_{|\mu|}[B_{\mu}],
\end{equation}
hence in particular
\begin{equation} \label{eq:qBmu_Tmu}
B_{(n)}=[n]_q\quad \text{ and } \quad T_{(n)}=q^{\binom{n}{2}},
\end{equation}

It is useful to introduce the so-called \emph{star scalar product} on $\Lambda$, given by setting
\[ \< p_{\lambda},p_{\mu} \>_*=(-1)^{|\mu|- \ell(|\mu|)}\prod_{i=1}^{\ell(\mu)}(1-q^{\mu_i})(1-t^{\mu_i}) z_{\mu}\chi(\lambda=\mu). \]

\medskip

For every symmetric function $f[X]$ and $g[X]$ we have (see \cite[Proposition~1.8]{Garsia-Haiman-Tesler-Explicit-1999})
\begin{equation}
\langle f,g\rangle_*= \langle \omega \phi f,g\rangle=\langle \phi \omega f,g\rangle
\end{equation}
where 
\begin{equation}
\phi f[X] \coloneq f[MX]\qquad \text{ for all } f[X]\in \Lambda.
\end{equation}

For every symmetric function $f[X]$ we set
\begin{equation}
f^*=f^*[X] \coloneq \phi^{-1}f[X]=f\left[\frac{X}{M}\right].
\end{equation}
Then for all symmetric functions $f,g,h$ we have
\begin{equation} \label{eq:hperp_estar_adjoint}
\langle h^\perp f,g\rangle_*=\langle h^\perp f,\omega \phi g\rangle=\langle f,h\omega \phi g\rangle=\langle f,\omega \phi ( (\omega h)^* \cdot g)\rangle =\langle f,  (\omega h)^* \cdot g \rangle_*,
\end{equation}
meaning the operator $h^\perp$ is the adjoint of multiplication by $(\omega h)^*$ with respect to the star scalar product. We will use these basic facts freely throughout this article.

\subsection{The symmetric functions $E_{n,k}$}

Given a variable $z$, observe that $e_n\left[X\frac{1-z}{1-q}\right]$ is a polynomial in $z$, hence it can be written as
\begin{equation} \label{eq:Enkdef}
e_n\left[X\frac{1-z}{1-q}\right]=\sum_{k=1}^n \frac{(z;q)_k}{(q;q)_k} E_{n,k}[X] 
\end{equation}
for uniquely determined $E_{n,k}[X]\in\Lambda^{(n)}$. These symmetric functions were first introduced in \cite{Garsia-Haglund-qtCatalan-2002}, and it is easy to see (cf. \cite[Section~1]{Garsia-Haglund-qtCatalan-2002}) that they satisfy the formula
\begin{equation} \label{eq:Enkformula}
E_{n,k}[X]= q^k\sum_{r=0}^kq^{\binom{r}{2}}\qbinom{k}{r}_q(-1)^re_n\left[X\frac{1-q^{-r}}{1-q}\right]\in\Lambda^{(n)}. 
\end{equation}
Being a polynomial in $z$, we can make the substitution $z\mapsto q^j$ in \eqref{eq:Enkdef}, getting immediately the formula
\begin{equation} \label{eq:enEnkformula}
e_n[X[j]_q]=\sum_{k=1}^n\qbinom{k+j-1}{k}E_{n,k}[X]. 
\end{equation}

\subsection{Macdonald polynomials: basic properties}

For a partition $\mu\vdash n$, we denote by
\begin{equation}
\widetilde{H}_{\mu} \coloneq \widetilde{H}_{\mu}[X]=\widetilde{H}_{\mu}[X;q,t]=\sum_{\lambda\vdash n}\widetilde{K}_{\lambda, \mu}(q,t)s_{\lambda}
\end{equation}
the \emph{(modified) Macdonald polynomials}, where 
\begin{equation}
\widetilde{K}_{\lambda, \mu} \coloneq \widetilde{K}_{\lambda, \mu}(q,t)=K_{\lambda, \mu}(q,1/t)t^{n(\mu)}\quad \text{ with }\quad n(\mu)=\sum_{i\geq 1}\mu_i(i-1)
\end{equation}
are the \emph{(modified) Kostka coefficients} (see \cite[Chapter~2]{Haglund-Book-2008} for more details).

Recall the normalization given for all $k\in\NN$ by
\begin{equation} \label{eq:MacNormal}
h_k^\perp \wH_\mu=\widetilde{K}_{(k), \mu}(q,t)=1\quad \text{for all }\mu\vdash k.
\end{equation} 

It turns out that the Macdonald polynomials are orthogonal with respect to the star scalar product: more precisely
\begin{align}
\label{eq:H_orthogonality}
\< \widetilde{H}_{\lambda},\widetilde{H}_{\mu}\>_*=w_{\mu}(q,t)\chi(\lambda=\mu).
\end{align}
These orthogonality relations give the following Cauchy identity:
\begin{align}
\label{eq:Mac_Cauchy}
e_n \left[ \frac{XY}{M} \right] = \sum_{\mu \vdash n} \frac{ \widetilde{H}_{\mu} [X] \widetilde{H}_\mu [Y]}{w_\mu} \quad \text{ for all } n.
\end{align}

The following basic identity is well known.
\begin{equation} \label{eq:Hnhn}
\widetilde{H}_{(n)} = (q;q)_n h_n\left[\frac{X}{1-q}\right].
\end{equation}
The following identity is well known and it is proved in Corollary~\ref{cor:ehMacId}.
\begin{proposition}
Given $n,k\in\NN$, $n\geq k$, for every $\mu\vdash n$ we have
\begin{equation} \label{eq:hkstarenmkstar1}
\<h_k^*e_{n-k}^*,\wH_\mu\>_*=e_k[B_\mu] .
\end{equation}
\end{proposition}

We define the \emph{nabla} operator on $\Lambda$ by setting
\begin{equation}
\nabla  \widetilde{H}_{\mu}=(-1)^{|\mu|}T_{\mu} \widetilde{H}_{\mu}\quad \text{ for all }\mu.
\end{equation}
Notice that this definition differs from the usual definition from \cite{Bergeron-Garsia-Haiman-Tesler-Positivity-1999} by a sign, but it is in agreement with the convention in \cite{Garsia_Mellit} and \cite{DAdderio_Mellit2020}.

We introduce the multiplicative involution $\overline{\omega}$ defined on any symmetric function $F[X;q,t]\in\Lambda_{\mathbb{Q}(q,t)}$ as
\[ \overline{\omega}  F[X;q,t]\coloneq \omega F[X;1/q,1/t]. \]
The following well-known identity can be deduced from Macdonald's duality (see \cite[Theorem~2.7]{Garsia-Haiman-qLagrange-1996})
 \begin{equation} \label{eq:nablaomegabar}
\nabla  \overline{\omega}\wH_\mu[X;q,t] =(-1)^{|\mu|} \wH_\mu[X;q,t],
\end{equation}
and it implies that $\nabla \overline{\omega}$ is an involution.

\subsection{Delta and Theta operators}

We define the \emph{Delta} operators $\Delta_f$ and $\Delta_f'$ on $\Lambda$ by
\begin{equation}
\Delta_f \widetilde{H}_{\mu}=f[B_{\mu}(q,t)]\widetilde{H}_{\mu}\quad \text{ and } \quad 
\Delta_f' \widetilde{H}_{\mu}=f[B_{\mu}(q,t)-1]\widetilde{H}_{\mu},\quad \text{ for all }\mu.
\end{equation}
Observe that on the vector space of symmetric functions homogeneous of degree $n$, denoted by $\Lambda^{(n)}$, the operator $\nabla$ equals $(-1)^n\Delta_{e_n}$. Moreover, by the Pieri rule, for every $1\leq k\leq n$,
\begin{align}
\label{eq:deltaprime}
\Delta_{e_k} = \Delta_{e_k}' + \Delta_{e_{k-1}}' \quad \text{ on } \Lambda^{(n)},
\end{align}
and for any $k > n$, $\Delta_{e_k} = \Delta_{e_{k-1}}' = 0$ on $\Lambda^{(n)}$, so that $\Delta_{e_n}=\Delta_{e_{n-1}}'$ on $\Lambda^{(n)}$.

Observe that \eqref{eq:hkstarenmkstar1} can be rephrased as
\begin{equation} \label{eq:hkstarenmkstar}
\Delta_{e_k} e_n^* = h_k^*e_{n-k}^*\quad \text{or}\quad \<f,e_kh_{n-k}\> = \<\Delta_{e_k}f,h_n\>\quad \text{for every }f\in\Lambda^{(n)}.
\end{equation}

Recall the linear operator $\PPi$, defined by setting for any nonempty partition $\mu$
\begin{equation}
\PPi  \widetilde{H}_\mu \coloneq \Pi_\mu \widetilde{H}_\mu .
\end{equation}
The following property is easy to check (cf.\ \cite[Section~3]{Garsia-Hicks-Stout-2011}): for $n\geq 1$ and $F\in\Lambda^{(n)}$ we have
\begin{align} \label{eq:PiomegabarPi}
\overline{\omega} \PPi  \overline{\omega} F& = -\nabla^{-1}  \PPi F. 
\end{align}
It is well known and easy to show (cf.\ \cite[Proposition~2.3]{Garsia-Hicks-Stout-2011}) that 
\begin{equation} \label{eq:Pien}
\PPi e_n^* = \omega p_n^* = \frac{1}{M}\alpha_n p_n\quad\text{ where }\quad \alpha_n\coloneq(-1)^{n-1}/([n]_q[n]_t),
\end{equation}
\begin{equation} \label{eq:Pienen}
\Delta_{e_1}\PPi e_n^*=\frac{1}{M}e_n,
\end{equation}
\begin{equation} \label{eq:Pienhn}
(-qt)^{1-n}\Delta_{e_n}^{-1}\Delta_{e_{n-1}}\PPi e_n^*=\frac{1}{M}h_n,
\end{equation}
and
\begin{equation} \label{eq:omegapn}
\omega(p_n)=[n]_q[n]_tM\PPi e_n^*=\sum_{k=1}^k\frac{[n]_q}{[k]_q}E_{n,k}.
\end{equation}

Given any symmetric function $f\in\Lambda$, we denote by $\underline{f}$ the multiplication operator \[\underline{f}g\coloneq fg \quad \text{for all }g\in\Lambda.\]

For any symmetric function $f\in \Lambda^{(n)}$ we introduce the following \emph{Theta operators} on $\Lambda$: for every $F\in \Lambda^{(m)}$ we set
\begin{equation} \label{eq:def_Deltaf}
\Theta_fF  \coloneq \left\{\begin{array}{ll} 0 & \text{if }n\geq 1\text{ and }m=0 \\
f\cdot F & \text{if }n=0\text{ and }m=0 \\ \PPi f^*\PPi ^{-1}F  & \text{otherwise} \end{array}\right. ,
\end{equation}
and we extend by linearity the definition to any $f,F\in \Lambda$.

It is clear that $\Theta_f$ is linear, and moreover, if $f$ is homogeneous of degree $k$, then so is $\Theta_f$, i.e. \[\Theta_f\Lambda^{(n)}\subseteq \Lambda^{(n+k)} \qquad \text{ for }f\in \Lambda^{(k)}. \]

We observe here that, using \eqref{eq:Pienen}, for every $n\in\NN$ we have
\begin{equation} \label{eq:Thetaene1}
\Theta_{e_n}e_1=M\PPi e_n^*e_1^*=M\Delta_{e_1}\PPi e_{n+1}^*=e_{n+1}.
\end{equation}

\subsection{Tesler's identity and the five-term relation}

We introduce the plethystic exponential
\[ \Exp[X]\coloneq \sum_{n\geq 0}h_n[X], \]
and the translation and multiplication operators $\T_Y$ and $\Pc_Z$ for any two expressions $Y$ and $Z$ by setting for any $F[X]\in\Lambda$
\[ \T_Y F[X]\coloneq F[X+Y]\quad \text{and}\quad \Pc_ZF[X]\coloneq\Exp[ZX]F[X]. \]
Observe that $\T_Y^{-1}=\T_{-Y}$ and $\Pc_Z^{-1}=\Pc_{-Z}$, where the minus sign is the plethystic one.

Note that, following \cite{Garsia-Haiman-Tesler-Explicit-1999}, for any two expressions $Y$ and $Z$,
\begin{align*}
	\T_Y \Pc_Z  F[X] & = \T_Y   \Exp[XZ] F[X] \\
	& =  \Exp[(X+Y)Z] F[X+Y] \\
	& = \Exp[YZ] \Exp[XZ] F[X+Y] \\
	& = \Exp[YZ] \Pc_Z \T_Y  F[X].
\end{align*}
Therefore, for any two expressions $Y$ and $Z$, we have
\begin{equation}\label{translationmultiplication} 
	\T_Y \Pc_Z = \Exp[YZ] \Pc_Z \T_Y.
\end{equation}

The following formulas are proved in \cite[Theorem~1.1]{Garsia-Haiman-Tesler-Explicit-1999}: 
\begin{align}
\label{eq:defTY}\T_Y & = \sum_{\mu}s_\mu[Y]s_\mu^\perp\\
\label{eq:defPZ} \Pc_Z& = \sum_{\mu}s_\mu[Z]\underline{s}_\mu.
\end{align} As special cases, for any monomial $u$, we have 
\begin{align}
\label{eq:defTu}\T_u & = \sum_{ k \geq 0 }u^k h_k^\perp\\
\label{eq:defPu} \Pc_{ -\frac{u}{M}} & = \sum_{k \geq 0} (-u)^k  \underline{e}_k^*.
\end{align}
It will be convenient for us to use the modified star scalar product, defined in \cite{Garsia_Mellit} by setting
\begin{equation}
\langle f,g\rangle_{\bar{*}} := \langle  f[-MX] ,g\rangle.
\end{equation}
Under this scalar product, the modified Macdonald basis remains orthogonal and any linear operator which acts diagonally on the Macdonald basis is self dual. The dual of $h_k^{\perp}$ becomes $(-1)^k \underline{e}_{k}^*$, and we see that $\T_u$ is dual to $\Pc_{-u/M}$.

Finally, we introduce the operators
\begin{align*}
	& \Delta_v\coloneq \sum_{n\geq 0}(-v)^n\Delta_{e_n}, \qquad \Delta_v^{-1}=\sum_{n\geq 0}v^n\Delta_{h_n}.
\end{align*}

The following identities are proved in \cite[Theorem~1.1]{Garsia_Mellit}.
\begin{theorem}[Five-term relations] \label{thm:5trel}
For any two monomials $u$ and $v$ we have
\begin{align} \label{eq:5trel}
\nabla^{-1} \T_{uv} \nabla = \Delta_v^{-1} \T_u \Delta_v \T_{-u}
\end{align}
and its dual (with respect to the modified star scalar product)
\begin{align} \label{eq:dual5trel}
\nabla \Pc_{-\frac{uv}{M}} \nabla^{-1} = \Pc_{\frac{u}{M}} \Delta_v \Pc_{-\frac{u}{M}} \Delta_v^{-1}.
\end{align}
\end{theorem}
The following result is a combination of Proposition~2.4 and Theorem~2.8 in \cite{Garsia_Mellit}.
\begin{theorem}	
For any symmetric function $F$, 
\begin{equation} \label{eq:SinverseDeltaprime}
\nabla^{-1} \T_u \nabla \Pc_{-\frac{1}{M}} \T_{1} F =  \Pc_{-\frac{1}{M}} \T_{1}  \Delta'_u F.
\end{equation}
where \[\Delta'_u \coloneq  \Exp[u/M] \Delta_u.\]
\end{theorem} 
The following identity is proved in \cite[Theorem~I.2]{Garsia-Haiman-Tesler-Explicit-1999}.
\begin{theorem}[Tesler's identity] \label{thm:Teslersidentity}
For any monomial $z$ and any partition $\mu$ we have
\begin{align} \label{eq:TeslerId}
\T_{-\frac{1}{z}} \Pc_{\frac{z}{M}}\nabla^{-1}\Exp\left[\frac{- zXD_\mu}{M}\right] = \wH_\mu[zX].
\end{align}
\end{theorem}

\section{A new general identity}

Consider the following operators:
\begin{align*}
& \tTheta(z,v)\coloneq \Delta_v \Pc_{-\frac{z}{M}} \Delta_v^{-1}, \qquad  \tTheta(z,v)^{-1}=\Delta_v \Pc_{\frac{z}{M}} \Delta_v^{-1}.
\end{align*}
Observe that the coefficient of $z^k$ in $\tTheta(z,v)$ is $(-1)^k\Delta_v \underline{e}_k^* \Delta_v^{-1}$. It makes sense to let $v$ go to $1$ (the factor $(1-v)^{-1}$ from $\Delta_v^{-1}$ cancels with the factor $(1-v)$ from $\Delta_v$), and this gives precisely our definition of $\Theta_{e_k}$ multiplied by $(-1)^k$.

The following theorem is the main result of this article.
\begin{theorem} \label{thm:mainidentity} We have
\begin{equation} \label{eq:mainidentity}
\tTheta(z,v)^{-1} \T_u \tTheta(z,v) = \Exp\left[\frac{uz(v-1)}{M}\right]\Delta_{uzv} \T_u.
\end{equation}
\end{theorem}
\begin{proof}
In order to keep track of homogeneous degrees in \eqref{eq:SinverseDeltaprime}, we insert a variable $w$ to get
\[\nabla^{-1} \T_{u/w} \nabla \Pc_{-\frac{w}{M}} \T_{1/w}  =  \Pc_{-\frac{w}{M}} \T_{1/w}  \Delta'_u .\]
Making the replacement $u \mapsto uw$ and rearranging, we get
\begin{equation} \label{eq:nablam1Tunabla}
	\nabla^{-1} \T_{u} \nabla  =  \Pc_{-\frac{w}{M}} \T_{1/w}  \Delta'_{uw}  \T_{-1/w} \Pc_{\frac{w}{M}}  .
\end{equation}
We now have
\begin{align*}
\tTheta(z,v)^{-1} \T_u \tTheta(z,v) & =   \left( \Delta_v \Pc_{\frac{z}{M} }\Delta_v^{-1} \right) \T_u \left( \Delta_{v} \Pc_{-\frac{z}{M}} \Delta_{v}^{-1} \right) 
\\   {\text{(using \eqref{eq:dual5trel})}} &=  \nabla \Pc_{\frac{zv}{M} } \nabla^{-1}   \left( \Pc_{\frac{z}{M} }  \T_u \Pc_{- \frac{z}{M} }  \right) \nabla \Pc_{-\frac{zv}{M}} \nabla^{-1} 
\\
{ \text{(using \eqref{translationmultiplication})} }&  =  \Exp\left[-\frac{uz}{M} \right]  \nabla \Pc_{\frac{zv}{M} } \left( \nabla^{-1}   \T_u   \nabla \right) \Pc_{-\frac{zv}{M}} \nabla^{-1} 
\\ {\text{(using \eqref{eq:nablam1Tunabla})} }& =  
\Exp\left[-\frac{uz}{M} \right]  \nabla \Pc_{\frac{zv}{M} }     \Pc_{-\frac{w}{M}} \T_{1/w}  \Delta'_{uw}  \T_{-1/w} \Pc_{\frac{w}{M}} 
\Pc_{-\frac{zv}{M}} \nabla^{-1} 
\\ {\text{(replacing $w\mapsto zv$)}} & =  
\Exp\left[-\frac{uz}{M} \right] \Exp\left[\frac{uzv}{M}\right]  \nabla  \T_{1/zv}  \Delta_{uzv}  \T_{-1/zv}\nabla^{-1} \\
\text{(using \eqref{eq:5trel})}& = 
\Exp\left[\frac{uz(v-1)}{M} \right]  \nabla \Delta_{uzv} \nabla^{-1} \T_{u}  
\\
& = 
\Exp\left[\frac{uz(v-1)}{M} \right]   \Delta_{uzv}  \T_{u}  ,
\end{align*}
which is what we wanted.
\end{proof}

In the rest of this article we show what can be deduced from this general identity.

\section{Reciprocities}

Unsurprisingly, combining Tesler's identity \eqref{eq:TeslerId} with our main identity \eqref{eq:mainidentity}, we can easily recover reciprocity identities. They will be given as applications of the following result.
\begin{theorem} \label{thm:TeslerTheta}  
	For any $k\in\mathbb{N}$ and any $\mu\vdash k$ we have
	\begin{equation}  \label{eq:vThetareciprocity}
	\left.\Exp[uz(1-v)/M]\T_u \tTheta(z,v)\widetilde{H}_\mu[X] \right|_{u^k}= \Exp\left[ \frac{- z X(vD_\mu+1)}{M}\right]. 
	\end{equation}
\end{theorem}
\begin{proof}
	Our main identity \eqref{eq:mainidentity} gives
	\begin{align*}
	\Exp[uz(1-v)/M]\T_u \tTheta(z,v) & = \tTheta(z,v)\Delta_{uzv}\T_u\\
	& = \Delta_v \Pc_{-\frac{z}{M}} \Delta_v^{-1}\Delta_{uzv}\T_u\\
	\text{(using \eqref{eq:dual5trel})}	& = \Pc_{-\frac{z}{M}} \nabla \Pc_{-\frac{zv}{M}} \nabla^{-1}\Delta_{uzv}\T_u.
	\end{align*}
	Extracting the coefficient of $u^k$ from the last expression we get
	\begin{align*}
	\Pc_{-\frac{z}{M}} \nabla \Pc_{-\frac{zv}{M}} \nabla^{-1}\sum_{r=0}^k(-1)^{k-r}(zv)^{k-r} \Delta_{e_{k-r}}h_r^\perp.
	\end{align*}
	But observe that for any element $F\in \Lambda^{(k)}$ (the symmetric functions of homogeneous degree $k$), $h_r^\perp F\in \Lambda^{(k-r)}$, hence
	\[ \nabla^{-1}(-1)^{k-r}\Delta_{e_{k-r}}h_r^\perp F = \nabla^{-1}\nabla h_r^\perp F= h_r^\perp F.\]
	Therefore, on $\Lambda^{(k)}$, our last expression can be written as 
	\[  \Pc_{- \frac{z}{M}} \nabla \Pc_{-\frac{zv}{M}} \sum_{r=0}^k(zv)^{k-r} h_r^\perp= \Pc_{-\frac{z}{M}} \nabla \Pc_{- \frac{zv}{M}} \T_{1/zv}(zv)^k.\]
	Now using Tesler's identity \eqref{eq:TeslerId}, for any $\mu\vdash k$ we get
	\begin{align*}
	\Pc_{- \frac{z}{M}} \nabla \Pc_{- \frac{zv}{M}} \T_{1/zv}(zv)^k\wH_\mu[X]
	& = \Pc_{- \frac{z}{M}}\Exp\left[ \frac{- zvXD_\mu}{M}\right]=\Exp\left[ \frac{- z X(vD_\mu+1)}{M}\right].
	\end{align*}
\end{proof}
The following result is extremely useful.
\begin{theorem} \label{thm:Thetareciprocity}
	For any $k,m\in\mathbb{N}$, and $\mu\vdash k$ we have
	\begin{equation}  \label{eq:Thetareciprocity}
	h_k^\perp \Theta_{e_m}\widetilde{H}_\mu=e_m[XB_\mu]. 
	\end{equation}
\end{theorem}
\begin{proof}
	Just make the replacement $v\mapsto 1$ in \eqref{eq:vThetareciprocity} and extract the coefficient of $z^m$.
\end{proof}

\subsection{Some applications}

The following corollary is a version of the Macdonald-Koornwinder reciprocity that first appeared in \cite[Theorem~3.3]{Garsia-Haiman-Tesler-Explicit-1999}.
\begin{theorem}
	For any two partitions $\lambda\vdash m$ and $\mu\vdash k$ we have
	\begin{equation} \label{eq:vMacRec}
	\wH_\lambda[vD_\mu+1]\prod_{(i,j)\in \mu}(1-vq^it^j)=\wH_\mu[vD_\lambda+1] \prod_{(i,j)\in \lambda}(1-vq^it^j).
	\end{equation}
\end{theorem}
\begin{proof}
	We have
	\begin{align*}
	&\wH_\lambda[vD_\mu+1]\prod_{(i,j)\in \mu}(1-vq^it^j) =\\
	\text{(using \eqref{eq:Mac_Cauchy})}& = \left\<\Exp\left[ \frac{-  X(vD_\mu+1)}{M}\right],\wH_\lambda[X]\right\>_{\bar{*}}\prod_{(i,j)\in \mu}(1-vq^it^j)\\
	& = \left.\left\<\Exp\left[ \frac{- z X(vD_\mu+1)}{M}\right],\wH_\lambda[X]\right\>_{\bar{*}}\right|_{z^m}\prod_{(i,j)\in \mu}(1-vq^it^j)\\
	\text{(using \eqref{eq:vThetareciprocity})} & = \left.\left\<\Exp[uz(1-v)/M]\T_u \tTheta(z,v)\Delta_v\widetilde{H}_\mu[X],\wH_\lambda[X]\right\>_{\bar{*}}\right|_{u^kz^m}\\
	& =\left.\left\<\widetilde{H}_\mu[X],\Exp[uz(1-v)/M]\T_z \tTheta(u,v)\Delta_v\wH_\lambda[X]\right\>_{\bar{*}}\right|_{u^kz^m}\\
	\text{(using \eqref{eq:vThetareciprocity})} & =  \left.\left\<\widetilde{H}_\mu[X],\Exp\left[ \frac{- u X(vD_\lambda+1)}{M}\right]\right\>_{\bar{*}}\right|_{u^k}\prod_{(i,j)\in \lambda}(1-vq^it^j)\\
	& = \left\<\widetilde{H}_\mu[X],\Exp\left[ \frac{-X(vD_\lambda+1)}{M}\right]\right\>_{\bar{*}} \prod_{(i,j)\in \lambda}(1-vq^it^j)\\
	\text{(using \eqref{eq:Mac_Cauchy})} & = \widetilde{H}_\mu[vD_\lambda+1] \prod_{(i,j)\in \lambda}(1-vq^it^j).
	\end{align*}
\end{proof}
We limit ourselves to sketch here three nice well-known applications of this important result.
\begin{corollary}
For any partition $\mu$ we have
\begin{equation} \label{eq:H1mv}
\wH_\lambda[1-v] = \prod_{(i,j)\in \lambda}(1-vq^it^j).
\end{equation}
\end{corollary}
\begin{proof}
Just take $\mu$ to be the empty partition in \eqref{eq:vMacRec}.
\end{proof}
We can now prove \eqref{eq:hkstarenmkstar1}.
\begin{corollary} \label{cor:ehMacId}
	Given $n,k\in\NN$, $n\geq k$, for any partition $\lambda\vdash n$ we have
\begin{equation}
\<\wH_\lambda,s_{(n-k,1^k)}\>=e_k[B_\lambda-1],\text{ hence }\<\wH_\lambda,e_kh_{n-k}\>=e_k[B_\lambda].
\end{equation}
\end{corollary}
\begin{proof}
From \eqref{eq:H1mv} we get
\[ \wH_{\lambda}[1-v]=(1-v)\sum_{k=0}^{n-1} (-v)^ke_k[B_\lambda-1].  \]
On the other hand
\[ \wH_{\lambda}[1-v]=\sum_{\mu\vdash n}\<\wH_\lambda,s_\mu\>s_\mu[1-v]. \]
Now using \eqref{eq:slambda1mv} and comparing the polynomials in $v$ we get the first identity. The second one is an immediate consequence of the Pieri rule.
\end{proof}
The following well-known corollary is also immediate.
\begin{corollary}
	For nonempty partitions $\lambda\vdash n$ and $\mu\vdash k$ we have
	\begin{equation} \label{eq:MacRec}
	\Pi_\mu\wH_\lambda[MB_\mu] = \Pi_\lambda\wH_\mu[MB_\lambda]. 
	\end{equation}
\end{corollary}
\begin{proof}
Divide \eqref{eq:vMacRec} by $1-v$ and let $v\mapsto 1$.
\end{proof}

\section{More important consequences}

We deduce the following important identities from our main theorem.
\begin{corollary}
We have
\begin{align} 
\label{eq:TmuThetaz}\T_{u}^{-1}\tTheta(z,v) & = \Exp\left[\frac{uz(1-v)}{M}\right]\tTheta(z,v)\T_{u}^{-1} \Delta_{uzv}^{-1} \\
\label{eq:TuThetamz} \T_u \tTheta(z,v)^{-1}& = \Exp\left[\frac{uz(1-v)}{M}\right]\Delta_{uzv}^{-1}\tTheta(z,v)^{-1}\T_u
\end{align}
\end{corollary}
\begin{proof}
Identity \eqref{eq:TmuThetaz} is obtained from \eqref{eq:mainidentity} by taking the inverse formula
\[ \T_{u}  \tTheta(z,v)^{-1} \T_u^{-1} \tTheta(z,v) = \Exp\left[\frac{uz(1-v)}{M}\right]\Delta_{uzv}^{-1} ,\]
and composing on the left by $\tTheta(z,v)  \T_{u}^{-1}$. Similarly, \eqref{eq:TuThetamz} is obtained by composing on the right by $\tTheta(z,v)^{-1} \T_u$.
\end{proof}
The following theorem contains some of the main consequences of \eqref{eq:mainidentity}. 
\begin{theorem} \label{thm:mainconseq1}
	We have
	\begin{align}
	 \label{eq:hjek}
	 h_j^\perp\Theta_{e_k} & =\sum_{r=0}^j\Theta_{e_{k-j+r}}\Delta_{e_{j-r}}h_r^{\perp}\\
	 \label{eq:ejek} e_j^\perp \Theta_{e_k} & =\sum_{r=0}^j \Theta_{e_{k-j+r}} e_r^\perp \Delta_{h_{j-r}}\\
	  \label{eq:hjhk} h_j^\perp \Theta_{h_k} & =\sum_{r=0}^j\Delta_{h_{j-r}} \Theta_{h_{k-j+r}}h_r^\perp.
	  \end{align}
  \end{theorem}
\begin{proof}
Identity \eqref{eq:hjek} is obtained from \eqref{eq:mainidentity} by letting $v\mapsto 1$ and picking the coefficient of $u^jz^k$, getting
\[ (-1)^kh_j^\perp\Theta_{e_k}=\sum_{r=0}^j(-1)^{k-j+r} \Theta_{e_{k-j+r}}(-1)^{j-r}\Delta_{e_{j-r}}h_r^{\perp}, \]
which is just another way to write \eqref{eq:hjek}. Identities \eqref{eq:ejek} and \eqref{eq:hjhk} are deduced in a similar way from \eqref{eq:TmuThetaz} and \eqref{eq:TuThetamz} respectively.
\end{proof}
We can also deduce the ``inverse'' identities.
\begin{corollary}
	We have
	\begin{align}
	\label{eq:ThetamzTmu} \tTheta(z,v)^{-1}\T_{u}^{-1} & = \Exp\left[\frac{uz(1-v)}{M}\right]  \T_{u}^{-1} \Delta_{uzv}^{-1}\tTheta(z,v)^{-1} \\
	\label{eq:ThetamzTu} \tTheta(z,v)^{-1}\T_u & = \Exp\left[\frac{uz(v-1)}{M}\right]\Delta_{uzv}\T_u \tTheta(z,v)^{-1}\\
	\label{eq:ThetazTmu} \tTheta(z,v) \T_{u}^{-1} & = \Exp\left[\frac{uz(v-1)}{M}\right] \T_{u}^{-1}\tTheta(z,v)\Delta_{uzv}.
	\end{align}
\end{corollary}
\begin{proof}
Identities \eqref{eq:ThetamzTmu}, \eqref{eq:ThetamzTu} and \eqref{eq:ThetazTmu} are obtained simply by taking the inverses of \eqref{eq:mainidentity}, \eqref{eq:TmuThetaz} and \eqref{eq:TuThetamz} respectively.
\end{proof}
\begin{corollary} \label{cor:otheridentities}
	We have
\begin{align}
	  \label{eq:Thkej} \Theta_{h_k} e_j^\perp& =\sum_{r=0}^j(-1)^{j-r}e_r^{\perp}\Delta_{h_{j-r}} \Theta_{h_{k-j+r}}\\
	  \label{eq:Thkhj} \Theta_{h_k} h_j^\perp & =\sum_{r=0}^j(-1)^{j-r}\Delta_{e_{j-r}} h_r^\perp\Theta_{h_{k-j+r}}\\
	\label{eq:Tekej} \Theta_{e_k}e_j^\perp  & =\sum_{r=0}^j(-1)^{j-r}e_r^\perp \Theta_{e_{k-j+r}}\Delta_{e_{j-r}} .
	\end{align}
\end{corollary}
\begin{proof}
Identity \eqref{eq:Thkej} is obtained from \eqref{eq:ThetamzTmu} by letting $v\mapsto 1$ and picking the coefficient of $u^jz^k$, getting
\[ (-1)^j\Theta_{h_k} e_j^\perp=\sum_{r=0}^j (-1)^re_r^{\perp}\Delta_{h_{j-r}} \Theta_{h_{k-j+r}}, \]
which is just another way to write \eqref{eq:Thkej}. Identities \eqref{eq:Thkhj} and \eqref{eq:Tekej} are deduced in a similar way from \eqref{eq:ThetamzTu} and \eqref{eq:ThetazTmu} respectively.
\end{proof}

\section{Some consequences of Theorem~\ref{thm:Thetareciprocity}}

\emph{In the rest of this article we will make extensive use of the following two basic facts without further reference to them: (1) the fact that the adjoint of the multiplication by $\omega f^*$ with respect to the star scalar product is $f^\perp$ (cf.\ \eqref{eq:hperp_estar_adjoint}), and (2) that all the operators that are diagonal with respect to the Macdonald polynomials $\wH_\mu$ are self adjoint with respect to the star scalar product (cf.\ \eqref{eq:H_orthogonality}).}

\medskip

We start with a few easy applications of \eqref{eq:Thetareciprocity}. Since they will be crucial in all the other applications, we will call them theorems.
\begin{theorem} \label{thm:mainconseq2}
	For every partition $\mu\vdash k$ and every $F\in\Lambda^{(n)}$ we have
\begin{equation} \label{eq:ThetaMBmu}
\<h_{k}^\perp\Theta_{e_{n}} \wH_\mu , F\>_* =F[MB_\mu]. 
\end{equation}
\end{theorem}
\begin{proof}
	Using \eqref{eq:Thetareciprocity} and \eqref{eq:Mac_Cauchy} we have
\[\<h_{k}^\perp\Theta_{e_{n}} \wH_\mu , F\>_*=\<e_n[XB_\mu] , F\>_*=\sum_{\lambda\vdash n} \<\wH_\lambda[X], F\>_* \wH_\lambda[MB_\mu] /w_\lambda=F[MB_\mu]. \]
\end{proof}
\begin{theorem}
	Given $m,d\in\NN$, for any $A\in \Lambda^{(d)}$ and any $F\in\Lambda^{(k)}$ we have
	\begin{equation} \label{eq:AperpDeltaomegaA}
	A^\perp h_k^\perp \Theta_{e_m} F = h_k^\perp \Theta_{e_{m-d}}\Delta_{\omega A}F. 
	\end{equation}
\end{theorem}
\begin{proof}
	For any $\gamma\vdash k$ and any $\alpha\vdash m-d$, using \eqref{eq:ThetaMBmu}, we have
	\begin{align*}
	\<A^\perp h_k^\perp \Theta_{e_m} \wH_\gamma,\wH_\alpha\>_* &  = \< h_k^\perp \Theta_{e_m} \wH_\gamma,(\omega A)^*\wH_\alpha\>_*= (\omega A)[B_\gamma]  \wH_\alpha[MB_\gamma] \\
	& = (\omega A)[B_\gamma]  \<h_k^\perp \Theta_{e_{m-d}} \wH_\gamma,\wH_\alpha\>_*= \<h_k^\perp \Theta_{e_{m-d}} \Delta_{\omega A}\wH_\gamma,\wH_\alpha\>_*.
	\end{align*}
\end{proof}
\begin{theorem}
	Given $k,m,\ell\in\mathbb{N}$, $m\geq 1$, for any $G\in\Lambda^{(k)}$ and $F\in\Lambda^{(\ell)}$ we have
	\begin{equation}  \label{eq:newthm}
	h_{k+\ell}^\perp \Theta_{e_m}\Theta_FG=\Delta_F h_k^\perp \Theta_{e_m}G. 
	\end{equation}
\end{theorem}
\begin{proof}
	Using \eqref{eq:AperpDeltaomegaA}, for any $A \in\Lambda^{(m)}$ we have
	\begin{align*}
	\<h_{k+\ell}^\perp \Theta_{e_m}\Theta_FG,A\>_*& =  \< \PPi ^{-1}G,(\omega F)^\perp h_m^\perp \Theta_{e_{k+\ell}} \PPi A\>_*\\
\text{(using \eqref{eq:AperpDeltaomegaA})}	& = \< \PPi ^{-1}G,h_m^\perp \Theta_{e_{k}}\Delta_F \PPi A\>_* = \< \Delta_F h_k^\perp \Theta_{e_m}G, A\>_*.
	\end{align*}
\end{proof}

\begin{theorem}
	Given postive $m,k\in\NN$, for every $F\in\Lambda^{(k)}$ we have
	\begin{equation} \label{eq:DeltaFPiemstar}
	h_k^\perp \Theta_{e_m}\PPi F^* = \Delta_F \PPi e_m^*.
	\end{equation}
\end{theorem}
\begin{proof}
	For every $\lambda\vdash m$ we have
	\begin{align*}
	\<h_k^\perp \Theta_{e_m}\PPi F^*,\wH_\lambda\>_* & =\<F^*,h_m^\perp\Theta_{e_k}\PPi \wH_\lambda\>_*\\
	\text{(using \eqref{eq:ThetaMBmu})}& =\Pi_\lambda F[B_\lambda] \\
	\text{(using \eqref{eq:hkstarenmkstar})}& = \<\Delta_F\PPi e_m^*,\wH_{\lambda}\>_*.
	\end{align*}
\end{proof}
The following result is an easy extension of Theorem~\ref{thm:Thetareciprocity} which will be relevant in further applications.
\begin{corollary}
	For any $k,m,\ell\in\mathbb{N}$, $m\geq 1$, $\mu\vdash k$ and $F\in\Lambda^{(\ell)}$ we have
	\begin{equation} \label{eq:DeltaFem}
	h_{k+\ell}^\perp \Theta_{e_m}\Theta_F\widetilde{H}_\mu=\Delta_{F}e_m[XB_\mu]. 
	\end{equation}
\end{corollary}
\begin{proof}
	Combine \eqref{eq:newthm} and \eqref{eq:Thetareciprocity}.
\end{proof}

\subsection{Some applications}

We recast here in terms of Theta operators two applications of Macdonald reciprocity, both due to Haglund. The first one is \cite[Equation~(2.58)]{Haglund-Schroeder-2004}.

\begin{corollary}
For every non empty partition $\nu\vdash k$, $A\in\Lambda^{(d)}$ and non-constant $F\in\Lambda^{(m)}$ we have
\[ \sum_{\mu\vdash k+d}\Pi_{\mu}d_{\mu,\nu}^{A^*}F[MB_\mu] = \Pi_\nu(\Delta_{A}F)[MB_\nu]. \]
\end{corollary}
\begin{proof}
\begin{align*} 
\sum_{\mu\vdash k+d}\frac{\Pi_{\mu}}{\Pi_\nu}d_{\mu,\nu}^{A^*}F[MB_\mu] & =h_{k+d}^\perp\Delta_{\phi F}\Theta_A\wH_\nu \\
\text{(using \eqref{eq:AperpDeltaomegaA})}& =(\omega \phi F)^\perp h_{k+d}^\perp\Theta_{e_m}\Theta_A\wH_\nu \\
\text{(using \eqref{eq:newthm})}& =(\omega \phi F)^\perp \Delta_A h_{k}^\perp\Theta_{e_m} \wH_\nu \\
& = \<h_{k}^\perp \Theta_{e_m} \wH_\nu[X], \Delta_{A} F\>_*\\
\text{(using \eqref{eq:ThetaMBmu})} & = (\Delta_{A}F)[MB_\nu].  
\end{align*}
\end{proof}

\begin{lemma}
Given positive $n,k\in\NN$, for every $P\in\Lambda^{(n)}$ and $Q\in \Lambda^{(k)}$ we have
\begin{equation} \label{eq:HaglundRec}
h_k^\perp\Delta_{P}\PPi  Q^*=  h_n^\perp\Delta_{Q}\PPi  P^*.
\end{equation}
\end{lemma}
\begin{proof}
Using \eqref{eq:AperpDeltaomegaA} we have
\begin{align*} 
h_k^\perp\Delta_{ P}\PPi  Q^* & =(\omega P)^\perp  h_k^\perp\Theta_{e_n}\PPi  Q^*= \<  P^* ,h_k^\perp\Theta_{e_n}\PPi  Q^*\>_* \\
& = \< h_n^\perp \Theta_{e_k} \PPi   P^* , Q^*\>_* = (\omega Q)^\perp  h_n^\perp \Theta_{e_k} \PPi  P^* =  h_n^\perp \Delta_{Q} \PPi P^*.
\end{align*}
\end{proof}

The following result is \cite[Theorem~1]{Garsia_Haglund_Romero2019}.
\begin{corollary}
Given $n,k\in\NN$, for every $P\in\Lambda^{(n)}$ and $Q\in \Lambda^{(k)}$ we have
\[ \<\Delta_{ P}\alpha_kp_k, \omega Q\> = \<\Delta_{ Q}\alpha_np_n,\omega P\>. \]
\end{corollary}
\begin{proof}
Using \eqref{eq:Pien} and \eqref{eq:HaglundRec} we have
\begin{align*}
\<\Delta_{ P}\alpha_kp_k, \omega Q\> & = \<\Delta_{ P}M\PPi e_k^*,  Q^*\>_*= M\<e_k^*, \Delta_{ P}\PPi  Q^*\>_*\\
&  =  M h_k^\perp \Delta_{ P}\PPi  Q^* = Mh_n^\perp\Delta_{ Q}\PPi  P^*= \<\Delta_{ Q}\alpha_np_n, \omega P\>.
\end{align*}
\end{proof}
We give here a few applications of the last result. The following result is \cite[Corollary~1]{Garsia_Haglund_Romero2019} and first appeared in \cite{Haglund-Schroeder-2004}.
\begin{corollary}
	Given positive $n,k\in\NN$, for every $Q\in\Lambda^{(n)}$ we have
	\begin{equation} \label{eq:HaglundLem}
	\<\Delta_{e_{k-1}}e_n, \omega Q\> = \<\Delta_{ Q}e_k,h_k\>. 
	\end{equation}
\end{corollary}
\begin{proof}
	Using \eqref{eq:Pienen} we have
	\begin{align*}
		\<\Delta_{e_{k-1}}e_n, \omega Q\> & = M\<\Delta_{e_{k-1}e_1}\PPi e_n^*,  Q^*\>_*\\
		& = Mh_n^\perp \Delta_{e_{k-1}e_1}\PPi  Q^* \\
	\text{(using \eqref{eq:HaglundRec})}	&  = Mh_k^\perp \Delta_{ Q}\PPi (e_{k-1}e_1)^* \\
	\text{(using \eqref{eq:hkstarenmkstar})}	&  = Mh_k^\perp \Delta_{ Q}\Delta_{e_1}\PPi (e_{k})^* \\
	\text{(using \eqref{eq:Pienen})} & = \<\Delta_{ Q} e_k,h_k\>.
	\end{align*}
\end{proof}
We reproduce here a quick application of the last result, which first appeared in \cite[Theorem~4.27]{DAdderio-Iraci-VandenWyngaerd-Bible}, as we will use it below.
\begin{corollary}
Given $m,n,k\in\NN$ with $n\geq k$ we have
\begin{equation} \label{eq:scalarDeltaprimehh}
\<\Delta_{h_n}\Delta_{e_{m-k}}'e_{m+1},h_{m+1}\> = \<\Delta_{e_{m+n-k-1}}'e_{m+n},h_mh_n\>. 
\end{equation}
\end{corollary}
\begin{proof}
Using \eqref{eq:HaglundLem} we have
\begin{align*}
\<\Delta_{h_n}\Delta_{e_{m-k}}e_{m+1},h_{m+1}\> & = \<\Delta_{e_{m}}e_{m+n-k},e_n h_{m-k}\>\\
\text{(using \eqref{eq:hkstarenmkstar})}& = \<\Delta_{e_{m}e_{n}}e_{m+n-k},h_{m+n-k}\>\\
\text{(using \eqref{eq:HaglundLem})}& = \<\Delta_{e_{m+n-k-1}}e_{m+n},h_{m}h_{n}\>.
\end{align*}
Now using \eqref{eq:deltaprime} we can easily conclude.
\end{proof}
The following result is \cite[Corollary~2]{Garsia_Haglund_Romero2019}.
\begin{corollary}
	Given positive $n,k\in\NN$, for every $Q\in\Lambda^{(n)}$ we have
	\[ \<\Delta_{h_{k-1}}e_n, \omega Q\> = (-qt)^{k-1} \<\Delta_{ Q}h_k,e_k\>. \]
\end{corollary}
\begin{proof}
	Using \eqref{eq:Pienen} we have
	\begin{align*}
	\<\Delta_{h_{k-1}}e_n, \omega Q\> & = M\<\Delta_{h_{k-1}e_1}\PPi e_n^*,   Q^*\>_*\\
	& = Mh_n^\perp \Delta_{ h_{k-1} e_{1}}\PPi  Q^* \\
	\text{(using \eqref{eq:HaglundRec})}	&  = Mh_k^\perp \Delta_{ Q}\PPi (h_{k-1}e_1)^* \\
	\text{(using \eqref{eq:hkstarenmkstar})}	&  = Mh_k^\perp \Delta_{e_{k}}\Delta_{ Q}\Delta_{e_{k}}^{-1}\Delta_{e_{k-1}}\PPi (e_{k})^* \\
		\text{(using \eqref{eq:Pienhn})}	&  = (-qt)^{k-1}h_k^\perp \Delta_{e_{k}}\Delta_{ Q}h_k \\
     & = (-qt)^{k-1}\<\Delta_{e_k}\Delta_{ Q} h_k, h_k\>\\
	\text{(using \eqref{eq:hkstarenmkstar})} & = (-qt)^{k-1}\<\Delta_{ Q} h_k,e_k\>.
	\end{align*}
\end{proof}

\section{Some consequences of Corollary~\ref{cor:otheridentities}}

In this section we list some consequences of Corollary~\ref{cor:otheridentities}: several are new, others already appeared in the literature, in which case our new proofs are usually drastically shorter.
\begin{lemma} \label{lem:hkperpThetamFzero}
Given $m,\ell,k\in\NN$, with $m\geq 1$ and $k>\ell$, for any $F\in\Lambda^{(\ell)}$ we have
\[h_k^\perp \Theta_{e_m}F=0. \]
\end{lemma}
\begin{proof}
Using \eqref{eq:hjek} we have
\begin{align*}
h_k^\perp \Theta_{e_m}F & = \sum_{r=0}^k\Theta_{e_{m-k+r}}\Delta_{e_{k-r}}h_r^\perp F,
\end{align*}
but now $h_r^\perp F\in\Lambda^{(\ell-r)}$ so that $\Delta_{e_{k-r}}h_r^\perp F=0$ for all $r$.
\end{proof}

The following result first appeared in \cite[Lemma~6.1]{dadderio2019theta}.
\begin{theorem}
Given $n,k,\ell\in\NN$, $n\geq \ell$, for any $F\in\Lambda^{(n-\ell)}$
\begin{equation} \label{eq:lemThetaekDeltahk} 
\<\Theta_{e_\ell} F, h_k e_{n-k}\>= \<\Delta_{h_\ell} F, h_k e_{n-k-\ell}\>. 
\end{equation}
\end{theorem}
\begin{proof}
Using \eqref{eq:ejek} and \eqref{eq:hjek} we have
\begin{align*}
h_k^\perp  e_{n-k}^\perp \Theta_{e_\ell}  & = \sum_{r=0}^{n-k}h_k^\perp\Theta_{e_{\ell-n+k+r}}e_{r}^\perp \Delta_{h_{n-k-r}} = \sum_{r=0}^{n-k}\sum_{s=0}^k \Theta_{e_{\ell-n+r+s}}\Delta_{e_{k-s}}h_s^\perp e_{r}^\perp \Delta_{h_{n-k-r}}.
\end{align*}
Acting on a symmetric function of degree $n-\ell$, the terms of this sum can be nonzero only if $r+s=n-\ell$: we need $\ell-n+r+s \geq 0$ for the Theta operator and $s+r \leq n-\ell$ from applying $h_s^\perp e_{r}^\perp$. Now if this is the case, but $s<k$, then $\Delta_{e_{k-s}}$ would apply to a constant, giving zero. Hence the only remaining term is the one with $s=k$ and $r=n-k-\ell$, which is $h_k^\perp  e_{n-k-\ell}^\perp \Delta_{h_{\ell}}$, as claimed.
\end{proof}

\begin{lemma}
For $n,r\in\NN$, $r>1$ we have
\begin{align}
\label{eq:hrperppn} h_r^\perp p_n & = \delta_{r,n}\\
\label{eq:hrperpen} h_r^\perp e_n & = \delta_{r,1}e_{n-1}.
\end{align}
\end{lemma}
\begin{proof}
Identity \eqref{eq:hrperppn} follows immediately from the well-know formula
\( h_r=\sum_{\lambda\vdash r}p_\lambda/z_\lambda\),
while \eqref{eq:hrperpen} follows from Pieri rule.
\end{proof}

\begin{theorem}
	Given $k,m,n\in \mathbb{N}$, $n\geq 1$ we have
	\begin{align} \label{eq:hkThetaemen}
	h_{k}^\perp\Theta_{e_{m}} e_{n} 
	& = \Theta_{e_{m-k}}\Delta_{e_{k}}   e_{n} + \Theta_{e_{m-k+1}}\Delta_{e_{k-1}} e_{n-1} .
	\end{align}
\end{theorem}
\begin{proof}
Using \eqref{eq:hjek} we have
\begin{align*}
h_{k}^\perp\Theta_{e_{m}} e_{n} & = \sum_{r=0}^k\Theta_{e_{m-k+r}}\Delta_{e_{k-r}}h_r^\perp e_n\\
\text{(using \eqref{eq:hrperpen})}& =  \Theta_{e_{m-k}}\Delta_{e_{k}} e_n+\Theta_{e_{m-k+1}}\Delta_{e_{k-1}}e_{n-1}.
\end{align*}
\end{proof}
The following result first appeared in \cite[Theorem~3.1]{dadderio2019theta}.

\begin{theorem}
Given $n,k\in\NN$, $n>k$ we have
	\[ \Delta_{e_{n-k}}e_n = \Theta_{e_k}\Delta_{e_{n-k}} e_{n-k} + \Theta_{e_{k+1}}\Delta_{e_{n-k-1}}e_{n-k-1} \]hence
	\begin{equation} \label{eq:Deltaprimeen}
	\Delta_{e_{n-k-1}}'e_n = \Theta_{e_k} \Delta_{e_{n-k}} e_{n-k}. 
	\end{equation}
\end{theorem}
\begin{proof}
Using \eqref{eq:Pienen}  we have
	\begin{align*}
	\Delta_{e_{n-k}}e_n & = M\Delta_{e_{n-k}e_1}\PPi e_n^*\\
	\text{(using \eqref{eq:DeltaFPiemstar})}& =Mh_{n-k+1}^\perp\Theta_{e_{n}} \PPi  e_{n-k}^*e_1^*\\
	\text{(using \eqref{eq:hkstarenmkstar})} & =Mh_{n-k+1}^\perp\Theta_{e_{n}} \Delta_{e_1}\PPi  e_{n-k+1}^*\\
	\text{(using \eqref{eq:Pienen})}& =h_{n-k+1}^\perp\Theta_{e_{n}} e_{n-k+1}\\
	 \text{(using \eqref{eq:hkThetaemen})}& = \Theta_{e_{k-1}}\Delta_{e_{n-k+1}} e_{n-k+1} + \Theta_{e_{k}}\Delta_{e_{n-k}} e_{n-k}.
	\end{align*}
	The last assertion follows now from \eqref{eq:deltaprime}.
\end{proof}

We take the chance to give a reformulation (using \eqref{eq:Pien}) of \cite[Theorem~3.3]{dadderio2019theta}.
\begin{theorem}
	Given $n,k\in\NN$ with $n>k$ we have
	\begin{equation}  \label{eq:Deltapn}
	\Delta_{e_{n-k}}\PPi e_n^* = \Theta_{e_k} \Delta_{e_{n-k}}  \PPi e_{n-k}^*. 
	\end{equation}
\end{theorem}
\begin{proof}
	Using \eqref{eq:hkstarenmkstar}	we have
	\[ \Delta_{e_{n-k}}\PPi e_n^* = \PPi (h_{n-k}^*e_k^*)=\Theta_{e_k}\PPi h_{n-k}^*=\Theta_{e_k} \Delta_{e_{n-k}} \PPi e_{n-k}^*. \]
\end{proof}

\subsection{Some applications}

The following lemma will be useful.
\begin{lemma}
	Given $k,m,n\in \mathbb{N}$, $m,n\geq 1$ we have
	\begin{align} \label{eq:hkperpThetaem}
	h_{k}^\perp\Theta_{e_{m}} \PPi e_n^* 
	& = \PPi (h_k^*e_{m-k}^*e_{n-k}^*) .
	\end{align}
\end{lemma}
\begin{proof}
	If $k>n$, we get $0$ on both sides. If $k = n$, this is a simple consequence of \eqref{eq:DeltaFPiemstar}. Now using \eqref{eq:hjek}, we have for $k <n$,
	\begin{align*}
	h_{k}^\perp\Theta_{e_{m}} \PPi e_n^* & =\sum_{r=0}^k\Theta_{e_{m-k+r}}\Delta_{e_{k-r}}h_r^\perp \PPi e_n^*\\
	\text{(using \eqref{eq:Pien} and \eqref{eq:hrperppn})} & =\Theta_{e_{m-k}}\Delta_{e_{k}}\PPi e_n^*\\
	\text{(using \eqref{eq:hkstarenmkstar})} & =\Theta_{e_{m-k}}\PPi (h_k^*e_{n-k}^*)\\
	& = \PPi (h_k^*e_{m-k}^*e_{n-k}^*).
	\end{align*}
\end{proof}
The following result first appeared in \cite[Theorem~4.4]{DAdderio_Iraci_newdinv}.
\begin{corollary}
	Given $m,n,k\in\NN$ with $n\geq 1$ and $m> k$ we have
\[ \<\Delta_{e_{m+n-k}} e_{m+n-k},e_kh_{n-k}h_{m-k}\> = \<\Delta_{h_n}\Delta_{e_{m-k}}'e_{m+1},h_{m+1}\>. \]
\end{corollary}
\begin{proof}
We have
\begin{align*}
\<\Delta_{e_{m+n-k}} e_{m+n-k},e_kh_{n-k}h_{m-k}\> & = \<\PPi^{-1}\Delta_{e_{m+n-k}} e_{m+n-k},\PPi (h_k^*e_{n-k}^*e_{m-k}^*)\>_*\\
\text{(using \eqref{eq:hkperpThetaem})}& = \<\PPi^{-1}\Delta_{e_{m+n-k}} e_{m+n-k},h_k^\perp \Theta_{e_{m}}\PPi  e_n^*\>_*\\
 & = \<\Theta_{e_k}\Delta_{e_{m+n-k}} e_{m+n-k},e_{m}^*e_{n}^*\>_*\\
\text{(using \eqref{eq:Deltaprimeen})}& = \<\Delta_{e_{m+n-k-1}}' e_{m+n},h_{m}h_{n}\>\\
\text{(using \eqref{eq:scalarDeltaprimehh})} & = \<\Delta_{h_n}\Delta_{e_{m-k}}'e_{m+1},h_{m+1}\>.
\end{align*}
\end{proof}

The following theorem first appeared in \cite[Theorem~4.22 and Corollary 4.24]{DAdderio-Iraci-VandenWyngaerd-Bible}.
\begin{theorem}
Given positive $a,b,k\in\NN$ with $a\geq k$ we have
\begin{equation}
\<\Delta_{e_a}\Delta_{e_{a+b-k-1}}'e_{a+b},h_{a+b}\>= \<\Delta_{h_k}\Delta_{e_{a-k}}e_{a+b-k},e_{a+b-k}\>,
\end{equation}
hence
\begin{equation}
\<\Delta_{e_a}'\Delta_{e_{a+b-k-1}}'e_{a+b},h_{a+b}\>= \<\Delta_{h_k}\Delta_{e_{a-k}}'e_{a+b-k},e_{a+b-k}\>.
\end{equation}
\end{theorem}
\begin{proof}
	For the first identity we have
\begin{align*}
\<\Delta_{e_a}\Delta_{e_{a+b-k-1}}'e_{a+b},h_{a+b}\> & = \< \Delta_{e_{a+b-k-1}}'e_{a+b},e_ah_{b}\>\\
\text{(using \eqref{eq:Deltaprimeen})} & = \< \Theta_{e_k}\Delta_{e_{a+b-k}} e_{a+b-k},e_{a}h_{b}\>\\
\text{(using \eqref{eq:lemThetaekDeltahk})}& = \< \Delta_{h_k}\Delta_{e_{a+b-k}} e_{a+b-k},e_{a-k}h_{b}\>\\
\text{(using \eqref{eq:hkstarenmkstar})} & = \< \Delta_{h_k}\Delta_{e_{a-k}} e_{a+b-k},e_{a+b-k}\>.
\end{align*}
Now the second identity follows easily from \eqref{eq:deltaprime}: replace $a$ by $a-i$ and $b$ by $b+i$ in the first identity, and then sum over $i\in\NN$ with alternating signs.
\end{proof}

The following theorem is \cite[Theorem~2]{Garsia_Haglund_Romero2019}.
\begin{corollary}
	Given positive $n,k,j\in\NN$, $n\geq k$, $n>j$, we have
	\begin{equation} \label{eq:GHRthm2}
	h_j^\perp \Delta_{e_{n-k}}\alpha_n p_n = \Delta_{e_{n-k-j}h_j}\alpha_{n-j}p_{n-j}. 
	\end{equation}
\end{corollary}
\begin{proof}
	Using \eqref{eq:Pien} we have
	\begin{align*}
	h_j^\perp \Delta_{e_{n-k}}\alpha_n p_n & =Mh_j^\perp \Delta_{e_{n-k}}\PPi e_n^*\\
	\text{(using \eqref{eq:DeltaFPiemstar})}& =Mh_j^\perp h_{n-k}^\perp \Theta_{e_n} \PPi e_{n-k}^*\\
	\text{(using \eqref{eq:AperpDeltaomegaA})}& =Mh_{n-k}^\perp \Theta_{e_{n-j}} \Delta_{e_j}\PPi e_{n-k}^*\\
	\text{(using \eqref{eq:hkstarenmkstar})}& =Mh_{n-k}^\perp \Theta_{e_{n-j}} \PPi (h_j^*e_{n-k-j}^*)\\
	\text{(using \eqref{eq:DeltaFPiemstar})}& =M \Delta_{h_je_{n-k-j}}\PPi e_{n-j}^*\\
	\text{(using \eqref{eq:Pien})}& =\Delta_{h_je_{n-k-j}} \alpha_{n-j} p_{n-j}.
	\end{align*}
\end{proof}
The following theorem is \cite[Theorem~3]{Garsia_Haglund_Romero2019}.
\begin{corollary}
	Given positive $n,k,j\in\NN$, $n\geq k$, $n>j$, we have
	\[ h_j^\perp \Delta_{h_{n-k}}\alpha_n p_n =  h_{n-k}^\perp \Delta_{h_{j}} \Delta_{e_{n-k-j}}\alpha_{2n-k-j}p_{2n-k-j}. \]
\end{corollary}
\begin{proof}
	Using \eqref{eq:Pien} we have
	\begin{align*}
	h_j^\perp \Delta_{h_{n-k}}\alpha_n p_n & =Mh_j^\perp \Delta_{h_{n-k}}\PPi e_n^* \\
	\text{(using \eqref{eq:DeltaFPiemstar})} & =Mh_j^\perp h_{n-k}^\perp \Theta_{e_n} \PPi  h_{n-k}^* \\
	\text{(using \eqref{eq:hkstarenmkstar})}& =Mh_j^\perp h_{n-k}^\perp \Theta_{e_n}\Delta_{e_{n-k}}\PPi e_{n-k}^*\\
	\text{(using \eqref{eq:AperpDeltaomegaA})}& =Mh_{n-k}^\perp  h_j^\perp h_{n-k}^\perp   \Theta_{e_{2n-k}}\PPi e_{n-k}^*\\
	\text{(using \eqref{eq:AperpDeltaomegaA})}& =Mh_{n-k}^\perp  h_{n-k}^\perp   \Theta_{e_{2n-k-j}}\Delta_{e_j}\PPi e_{n-k}^*\\
	\text{(using \eqref{eq:hkstarenmkstar})}& =Mh_{n-k}^\perp  h_{n-k}^\perp   \Theta_{e_{2n-k-j}} \PPi (h_j^*e_{n-k-j}^*)\\
	\text{(using \eqref{eq:DeltaFPiemstar})}& =Mh_{n-k}^\perp \Delta_{h_j e_{n-k-j}} \PPi e_{2n-k-j}^*\\
	\text{(using \eqref{eq:Pien})}& =  h_{n-k}^\perp \Delta_{h_{j}} \Delta_{e_{n-k-j}}\alpha_{2n-k-j}p_{2n-k-j}.
	\end{align*}
\end{proof}

\section{Simplifying a long proof}

In this section we provide a new proof of \cite[Theorem~4.6]{DAdderio-Iraci-VandenWyngaerd-Bible} which is drastically shorter than the original one. In fact, we will reprove also some of the auxiliary results needed, as their proofs are much shorter as well.

\subsection{A useful lemma}

The following lemma will be useful.

\begin{lemma}
	For any $n,j\in\NN$ we have
	\begin{align} \label{eq:ejperpHn}
	e_j^\perp \widetilde{H}_{(n)} & = q^{\binom{j}{2}} \qbinom{n}{j}_q \widetilde{H}_{(n-j)} \\ \label{eq:hjperpHn}
	h_j^\perp \widetilde{H}_{(n)} & = \qbinom{n}{j}_q \widetilde{H}_{(n-j)} .
	\end{align}
\end{lemma}
\begin{proof}
Using \eqref{eq:Hnhn}, we get 
\begin{align*}
\T_u \wH_{(n)} & =(q;q)_n h_n\left[ \frac{X+u}{1-q} \right]\\
\text{(using \eqref{eq:SchurSummation})} & = \sum_{j=0}^n u^j h_{n-j}\left[\frac{X}{1-q} \right]  (q;q)_n h_j\left[ \frac{1}{1-q} \right]    \\
\text{(using \eqref{eq:h_q_prspec})} & =  \sum_{j=0}^n u^j (q;q)_{n-j} h_{n-j}\left[\frac{X}{1-q} \right] \qbinom{n}{j}_q.
\end{align*}
Taking the coefficient of $u^j$, we get the second identity \eqref{eq:hjperpHn}. A similar computation gives the first identity \eqref{eq:ejperpHn}.
\end{proof}

\subsection{Reformulating and recovering some known results}

We start with a lemma.
\begin{lemma}
For $k,j\in\NN$ with $j\geq 1$ and $k\geq j$ we have
\begin{equation}
\Theta_{e_{k-j}} \wH_{(j)} =\sum_{r=0}^j (-1)^{j-r}q^{r-jr}q^{\binom{r}{2}}\qbinom{j}{r}_qh_{r}^\perp\Theta_{e_{k}} \wH_{(r)}.
\end{equation}
\end{lemma}
\begin{proof}
We start by using \eqref{eq:hjek} to get
\begin{align*}
	h_r^\perp \Theta_{e_k}\wH_{(r)} & = \sum_{s=0}^r\Theta_{e_{k-r+s}}\Delta_{e_{r-s}}h_s^{\perp} \wH_{(r)} \\
\text{(using \eqref{eq:hjperpHn} and \eqref{eq:qBmu_Tmu})}	& = \sum_{s=0}^rq^{\binom{r-s}{2}}\qbinom{r}{r-s}_q\Theta_{e_{k-r+s}} \wH_{(r-s)} \\
	& = \sum_{s=0}^rq^{\binom{s}{2}}\qbinom{r}{s}_q\Theta_{e_{k-s}} \wH_{(s)} .
\end{align*}
Now multiplying by $(-1)^{j-r}q^{r-jr}q^{\binom{r}{2}}\qbinom{j}{r}_q$ and summing over $r$ gives
\begin{align*}
	&  \sum_{r=0}^j (-1)^{j-r}q^{r-jr}q^{\binom{r}{2}}\qbinom{j}{r}_qh_{r}^\perp\Theta_{e_{k}} \wH_{(r)}=\\
	& = \sum_{r=0}^j (-1)^{j-r}q^{r-jr}q^{\binom{r}{2}}\qbinom{j}{r}_q\sum_{s=0}^rq^{\binom{s}{2}}\qbinom{r}{s}_q\Theta_{e_{k-s}} \wH_{(s)} \\
	& = \sum_{s=0}^j \sum_{r=s}^jq^{\binom{s}{2}}q^{-\binom{j}{2}} (-1)^{j-r}q^{\binom{j}{2}+(r-jr)+\binom{r}{2}}\qbinom{j}{r}_q\qbinom{r}{s}_q\Theta_{e_{k-s}} \wH_{(s)} \\
	& = \sum_{s=0}^j q^{\binom{s}{2}-\binom{j}{2}}\qbinom{j}{s}_q\sum_{r=s}^j (-1)^{j-r}q^{\binom{j-r}{2}}\qbinom{j-s}{r-s}_q\Theta_{e_{k-s}} \wH_{(s)} \\
\text{(using \eqref{eq:qbinom})}	& = \sum_{s=0}^j q^{\binom{s}{2}-\binom{j}{2}}\qbinom{j}{s}_q\delta_{j,s}\Theta_{e_{k-s}} \wH_{(s)} \\
	& = \Theta_{e_{k-j}} \wH_{(j)} 
\end{align*}
as we wanted.
\end{proof}
Applying \eqref{eq:newthm} we get immediately the following corollary, which is a reformulation of  \cite[Proposition~2.6]{Garsia-Hicks-Stout-2011}.
\begin{corollary}
	For $k,j\in\NN$ with $j\geq 1$ and $k\geq j$ we have
	\begin{equation} \label{eq:thetaHj}
	\Theta_{e_{k-j}} \wH_{(j)} =\sum_{r=0}^j (-1)^{j-r}q^{r-jr}q^{\binom{r}{2}}\qbinom{j}{r}_q e_k[X[r]_q].
	\end{equation}
\end{corollary}

The following theorem is due to Haglund \cite[Equation~(7.86)]{Haglund-Book-2008}.
\begin{theorem}
	Given positive $k,j\in\NN$ we have
	\begin{equation} \label{eq:nablaEnk}
	\Delta_{e_k} E_{k,j}=t^{k-j}\Theta_{h_{k-j}}\wH_{(j)}. 
	\end{equation}
\end{theorem}
\begin{proof}
We compute
\begin{align*}
&\overline{\omega}\sum_{r=0}^j (-1)^{j-r}q^{r-jr}q^{\binom{r}{2}}\qbinom{j}{r}_q e_k[X[r]_q]=\\
& =\sum_{r=0}^j (-1)^{j-r}q^{-r+jr}q^{-\binom{r}{2}}\qbinom{j}{r}_qq^{r^2-jr} \overline{\omega}e_k[X[r]_q]\\
& =\sum_{r=0}^j (-1)^{j-r} q^{r^2-r-\binom{r}{2}}\qbinom{j}{r}_q  h_k\left[X\frac{1-q^{-r}}{1-q^{-1}}\right]\\
& =\sum_{r=0}^j (-1)^{j-r} q^{\binom{r}{2}}\qbinom{j}{r}_q (-q)^k e_k\left[X\frac{1-q^{-r}}{1-q}\right]
\end{align*}
while we also have
\begin{align*}
\overline{\omega}	\Theta_{e_{k-j}} \wH_{(j)} & = \overline{\omega}\PPi  (e_{k-j}^* \PPi ^{-1}\wH_{(j)} )\\
& = \overline{\omega}\PPi  \overline{\omega}^2(e_{k-j}^* \PPi ^{-1}\wH_{(j)} )\\
& = \overline{\omega}\PPi  \overline{\omega}\left(  (qt)^{k-j}h_{k-j}^* \overline{\omega}\PPi ^{-1}\wH_{(j)} \right)\\
\text{(using \eqref{eq:PiomegabarPi})}& = (qt)^{k-j}\overline{\omega}\PPi  \overline{\omega}\left(  h_{k-j}^* (-1)\PPi ^{-1}\nabla \overline{\omega} \wH_{(j)} \right)\\
\text{(using \eqref{eq:nablaomegabar})}  & = (-1)^{j-1}(qt)^{k-j}\overline{\omega}\PPi  \overline{\omega}\left(  h_{k-j}^*\PPi ^{-1} \wH_{(j)} \right)\\
\text{(using \eqref{eq:PiomegabarPi})} & = (-1)^{j}(qt)^{k-j} \nabla^{-1} \PPi  \left(  h_{k-j}^*\PPi ^{-1} \wH_{(j)} \right)\\
& = (-1)^{j}(qt)^{k-j}\nabla^{-1}\Theta_{h_{k-j}}\wH_{(j)} ,
\end{align*}
so \eqref{eq:thetaHj} gives
\[ (-1)^{j}(qt)^{k-j}\nabla^{-1}\Theta_{h_{k-j}}\wH_{(j)}[X] = \sum_{r=0}^j (-1)^{j-r} q^{\binom{r}{2}}\qbinom{j}{r}_q (-q)^k e_k\left[X\frac{1-q^{-r}}{1-q}\right]\]
or better
\[  (-1)^k\nabla^{-1}t^{k-j}\Theta_{h_{k-j}}\wH_{(j)}[X]= q^j\sum_{r=0}^j (-1)^{r} q^{\binom{r}{2}}\qbinom{j}{r}_q  e_k\left[X\frac{1-q^{-r}}{1-q}\right].\]
But observe that the right hand side is precisely $E_{k,j}[X]$ (see \eqref{eq:Enkformula}), completing our proof.
\end{proof}
The following proposition first appeared in \cite[Proposition~4.9]{DAdderio-Iraci-VandenWyngaerd-Bible}.
\begin{proposition}
Given positive $k,j\in\NN$ we have
\begin{equation} \label{eq:nablaThetaejHi}
\Delta_{e_k} \Theta_{e_{k-j}} \wH_{(j)} =\sum_{s=1}^{k}q^{\binom{j}{2}}\qbinom{s-1}{j-1}_q t^{k-s}\Theta_{h_{k-s}}\wH_{(s)}.
\end{equation}
\end{proposition}
\begin{proof}
Applying $\Delta_{e_k}$ to \eqref{eq:thetaHj} we get
\begin{align*}
\Delta_{e_k} \Theta_{e_{k-j}} \wH_{(j)} & =\sum_{r=0}^j (-1)^{j-r}q^{r-jr}q^{\binom{r}{2}}\qbinom{j}{r}_q \Delta_{e_k} e_k[X[r]_q]\\
\text{(using \eqref{eq:enEnkformula})}&=\sum_{r=0}^j (-1)^{j-r}q^{r-jr}q^{\binom{r}{2}}\qbinom{j}{r}_q \sum_{s=1}^{k}\qbinom{r+s-1}{s}_q\Delta_{e_k} E_{k,s}[X]\\
\text{(using \eqref{eq:nablaEnk})}&=\sum_{s=1}^{k}\sum_{r=0}^j (-1)^{j-r}q^{r-jr}q^{\binom{r}{2}}\qbinom{j}{r}_q \qbinom{r+s-1}{s}_q t^{k-s}\Theta_{h_{k-s}}\wH_{(s)}\\
\text{(using \eqref{eq:lemmaelem})}&=\sum_{s=1}^{k}q^{\binom{j}{2}}\qbinom{s-1}{j-1}_q t^{k-s}\Theta_{h_{k-s}}\wH_{(s)}
\end{align*}
as we wanted.
\end{proof}

We are now able to give a much simpler proof of \cite[Theorem~4.6]{DAdderio-Iraci-VandenWyngaerd-Bible}, which can be reformulated as follows (cf.\ the corollary below). The special case $\ell=0$ was a theorem of Haglund \cite[Equation~(2.38)]{Haglund-Schroeder-2004}.
\begin{theorem}
Given $m,k,\ell\in\NN$, $m\geq 1$ we have
\begin{equation} \label{eq:mysummation}
h_{k+\ell}^\perp\Theta_{e_m}\Theta_{e_{\ell}}\widetilde{H}_{(k)} = \sum_{r=0}^{k}q^{\binom{r}{2}}\qbinom{k}{r}_q\sum_{b=1}^{\ell+r}\qbinom{k-r+b-1}{k-1}_q  t^{\ell+r-b}\Theta_{h_{\ell+r-b}}\Theta_{e_{m-\ell-r}}\widetilde{H}_{(b)}.
\end{equation}
\end{theorem}
\begin{proof}
Using \eqref{eq:hjek} we have
\begin{align*}
	h_{k+\ell}^\perp\Theta_{e_m}\Theta_{e_{\ell}}\widetilde{H}_{(k)} & = \sum_{r=0}^{k+\ell}\Theta_{e_{m-(k+\ell)+r}} \Delta_{e_{k+\ell-r}} h_r^\perp \Theta_{e_{\ell}}\widetilde{H}_{(k)}\\
	 & = \sum_{r=0}^{k+\ell}\Theta_{e_{m-(k+\ell)+r}} \Delta_{e_{k+\ell-r}} \sum_{s=0}^r\Theta_{e_{\ell-r+s}}\Delta_{e_{r-s}} h_s^\perp \widetilde{H}_{(k)}\\
\text{(using \eqref{eq:hjperpHn} and \eqref{eq:qBmu_Tmu})}	 & = \sum_{r=0}^{k+\ell}\Theta_{e_{m-(k+\ell)+r}} \Delta_{e_{k+\ell-r}} \sum_{s=0}^r\Theta_{e_{\ell-r+s}}q^{\binom{r-s}{2}}\qbinom{k-s}{r-s}_q \qbinom{k}{s}_q\widetilde{H}_{(k-s)}\\
	 & = \sum_{r=0}^{k+\ell}\Theta_{e_{m-(k+\ell)+r}}\sum_{s=0}^rq^{\binom{r-s}{2}}\qbinom{k-s}{r-s}_q \qbinom{k}{s}_q  \Delta_{e_{k+\ell-r}} \Theta_{e_{\ell-r+s}}\widetilde{H}_{(k-s)}.
\end{align*}
Now using \eqref{eq:nablaThetaejHi} we get
\begin{align*}
	& h_{k+\ell}^\perp\Theta_{e_m}\Theta_{e_{\ell}}\widetilde{H}_{(k)} =\\
	& = \sum_{r=0}^{k+\ell}\Theta_{e_{m-(k+\ell)+r}}\sum_{s=0}^rq^{\binom{r-s}{2}}\qbinom{k-s}{r-s}_q \qbinom{k}{s}_q \sum_{b=1}^{k+\ell-r}q^{\binom{k-s}{2}}\qbinom{b-1}{k-s-1}_qt^{k+\ell-r-b}\Theta_{h_{k+\ell-r-b}}\widetilde{H}_{(b)} \\
	& = \sum_{r=0}^{k+\ell}\Theta_{e_{m-(k+\ell)+r}}\sum_{b=1}^{k+\ell-r}\sum_{s=0}^rq^{\binom{r-s}{2}}\qbinom{r}{s}_q \qbinom{k}{r}_q q^{\binom{k-s}{2}}\qbinom{b-1}{k-s-1}_qt^{k+\ell-r-b}\Theta_{h_{k+\ell-r-b}}\widetilde{H}_{(b)} \\
	& = \sum_{r=0}^{k+\ell}q^{\binom{k-r}{2}}\qbinom{k}{k-r}_q\Theta_{e_{m-(k+\ell)+r}}\sum_{b=1}^{k+\ell-r}\sum_{s=0}^rq^{(k-s-1)(r-s)}\qbinom{r}{s}_q  \qbinom{b-1}{k-s-1}_qt^{k+\ell-r-b}\Theta_{h_{k+\ell-r-b}}\widetilde{H}_{(b)}  \\
	& = \sum_{r=0}^{k}q^{\binom{r}{2}}\qbinom{k}{r}_q\Theta_{e_{m-\ell-r}}\sum_{b=1}^{\ell+r}\sum_{s=0}^{k }q^{(k-s-1)(k-r-s)}\qbinom{k-r}{s}_q  \qbinom{b-1}{k-s-1}_qt^{\ell+r-b}\Theta_{h_{\ell+r-b}}\widetilde{H}_{(b)} \\
	& = \sum_{r=0}^{k}q^{\binom{r}{2}}\qbinom{k}{r}_q\Theta_{e_{m-\ell-r}}\sum_{b=1}^{\ell+r}\sum_{s=0}^{ k}q^{(s-1)(s-r)}\qbinom{k-r}{k-s}_q  \qbinom{b-1}{s-1}_qt^{\ell+r-b}\Theta_{h_{\ell+r-b}}\widetilde{H}_{(b)}\\
	& = \sum_{r=0}^{k}q^{\binom{r}{2}}\qbinom{k}{r}_q\Theta_{e_{m-\ell-r}}\sum_{b=1}^{\ell+r}\qbinom{k-r+b-1}{k-1}_q  t^{\ell+r-b}\Theta_{h_{\ell+r-b}}\widetilde{H}_{(b)},
\end{align*}
where in the last equality we used the well-known \eqref{eq:qChuVander}, and above we used the elementary
\begin{align*}
\binom{r-s}{2}+\binom{k-s}{2} 
& =\binom{k-r}{2}+(k-s-1)(r-s).
\end{align*}
\end{proof}
\begin{corollary}
	Given $m,k,\ell\in\NN$, $m\geq 1$ we have
	\begin{equation} \label{eq:corDeltaelemXk}
	\Delta_{e_{\ell}}e_m[X[k]_q] = \sum_{r=0}^{k}q^{\binom{r}{2}}\qbinom{k}{r}_q\sum_{b=1}^{\ell+r}\qbinom{k-r+b-1}{k-1}_q   \Theta_{e_{m-\ell-r}}\Delta_{e_{\ell+r}} E_{\ell+r,b} .
	\end{equation}
\end{corollary}
\begin{proof}
Just use \eqref{eq:DeltaFem} and \eqref{eq:nablaEnk} in \eqref{eq:mysummation} to conclude.
\end{proof}

\subsection{More applications}

The following result first appeared in \cite[Theorem~3.7]{DAdderio-Iraci-VandenWyngaerd-GenDeltaSchroeder}. The case $s=0$ already appeared in \cite[Proposition~4.26]{DAdderio-Iraci-VandenWyngaerd-Bible}.
\begin{theorem}
	Given $s,\ell,m,j\in\NN$, $m\geq 1$, we have
	\begin{equation}
		\sum_{k=1}^st^{s-k}\<\Delta_{h_{s-k}}\Delta_{e_{\ell}}e_m[X[k]_q],e_jh_{m-j}\> = \< \Delta_{h_j}\Delta_{e_{s-1}}'e_{s+\ell},e_{m-j}h_{s+\ell+j-m}\> .
	\end{equation}
\end{theorem}
\begin{proof}
	We have
	\begin{align*}
		\sum_{k=1}^st^{s-k}\<\Delta_{h_{s-k}}\Delta_{e_{\ell}}e_m[X[k]_q],e_jh_{m-j}\> & = \sum_{k=1}^st^{s-k}h_m^\perp \Delta_{e_j}\Delta_{h_{s-k}}\Delta_{e_{\ell}}e_m[X[k]_q] \\
		\text{(using \eqref{eq:DeltaFem})}& = \sum_{k=1}^st^{s-k}h_m^\perp h_{s+\ell+j}^\perp \Theta_{e_m}\Theta_{e_j}\Theta_{e_\ell}\Theta_{h_{s-k}}\wH_{(k)} \\
		\text{(using \eqref{eq:nablaEnk})}& = \sum_{k=1}^s h_m^\perp h_{s+\ell+j}^\perp \Theta_{e_m}\Theta_{e_j}\Theta_{e_\ell}\Delta_{e_s}E_{s,k} \\
		\text{(using \eqref{eq:enEnkformula})}& =  h_m^\perp h_{s+\ell+j}^\perp \Theta_{e_m}\Theta_{e_j}\Theta_{e_\ell}\Delta_{e_s}e_s \\
		\text{(using \eqref{eq:Deltaprimeen})}& =  h_m^\perp h_{s+\ell+j}^\perp \Theta_{e_m}\Theta_{e_j}\Delta_{e_{s-1}}'e_{s+\ell} \\
		\text{(using \eqref{eq:AperpDeltaomegaA})}& =  h_{s+\ell+j}^\perp \Delta_{e_m}\Theta_{e_j}\Delta_{e_{s-1}}'e_{s+\ell} \\
		\text{(using \eqref{eq:hkstarenmkstar})}& =  \< \Theta_{e_j}\Delta_{e_{s-1}}'e_{s+\ell},e_mh_{s+\ell+j-m}\>\\ 
		\text{(using \eqref{eq:lemThetaekDeltahk})}& =  \< \Delta_{h_j}\Delta_{e_{s-1}}'e_{s+\ell},e_{m-j}h_{s+\ell+j-m}\> .
	\end{align*}
\end{proof}
The following result first appeared in \cite[Theorem~4.7]{DAdderio-Iraci-VandenWyngaerd-Delta-Square}.
\begin{theorem}
	Given $s,\ell,m,j\in\NN$, $m\geq 1$, we have
	\begin{equation}
		\sum_{k=1}^s\frac{[s+\ell]_q}{[k]_q}t^{s-k}\<\Delta_{h_{s-k}}\Delta_{e_{\ell}}e_m[X[k]_q],e_jh_{m-j}\> =\frac{[s]_t}{[s+\ell]_t} \< \Delta_{h_j}\Delta_{e_{s}}\omega(p_{s+\ell}),e_{m-j}h_{s+\ell+j-m}\> .
	\end{equation}
\end{theorem}
\begin{proof}
	We have
	\begin{align*}
		& \sum_{k=1}^s\frac{[s+\ell]_q}{[k]_q}t^{s-k}\<\Delta_{h_{s-k}}\Delta_{e_{\ell}}e_m[X[k]_q],e_jh_{m-j}\>=\\
		& = \sum_{k=1}^s \frac{[s+\ell]_q}{[k]_q}t^{s-k}h_m^\perp \Delta_{e_j}\Delta_{h_{s-k}}\Delta_{e_{\ell}}e_m[X[k]_q] \\
		\text{(using \eqref{eq:DeltaFem})}& = \sum_{k=1}^s\frac{[s+\ell]_q}{[k]_q}t^{s-k}h_m^\perp h_{s+\ell+j}^\perp \Theta_{e_m}\Theta_{e_j}\Theta_{e_\ell}\Theta_{h_{s-k}}\wH_{(k)} \\
		\text{(using \eqref{eq:nablaEnk})}& = \sum_{k=1}^s \frac{[s+\ell]_q}{[k]_q}h_m^\perp h_{s+\ell+j}^\perp \Theta_{e_m}\Theta_{e_j}\Theta_{e_\ell}\Delta_{e_s}E_{s,k} \\
		\text{(using \eqref{eq:omegapn})}& = M[s+\ell]_q[s]_t  h_m^\perp h_{s+\ell+j}^\perp \Theta_{e_m}\Theta_{e_j}\Theta_{e_\ell}\Delta_{e_s}\PPi e_s^* \\
		\text{(using \eqref{eq:Deltapn})}& =  M[s+\ell]_q[s]_t h_m^\perp h_{s+\ell+j}^\perp \Theta_{e_m}\Theta_{e_j}\Delta_{e_{s}}\PPi e_{s+\ell}^* \\
		\text{(using \eqref{eq:omegapn})}& =  \frac{[s]_t}{[s+\ell]_t} h_m^\perp h_{s+\ell+j}^\perp \Theta_{e_m}\Theta_{e_j}\Delta_{e_{s}} \omega(p_{s+\ell}) \\
		\text{(using \eqref{eq:AperpDeltaomegaA})}& =  \frac{[s]_t}{[s+\ell]_t}h_{s+\ell+j}^\perp \Delta_{e_m}\Theta_{e_j}\Delta_{e_{s}} \omega(p_{s+\ell}) \\
		\text{(using \eqref{eq:hkstarenmkstar})}& =  \frac{[s]_t}{[s+\ell]_t}\< \Theta_{e_j}\Delta_{e_{s}} \omega(p_{s+\ell}),e_mh_{s+\ell+j-m}\>\\ 
		\text{(using \eqref{eq:lemThetaekDeltahk})}& =  \frac{[s]_t}{[s+\ell]_t}\< \Delta_{h_j}\Delta_{e_{s}} \omega(p_{s+\ell}),e_{m-j}h_{s+\ell+j-m}\> .
	\end{align*}
\end{proof}

\section{More new identities}
We can extend \eqref{eq:nablaThetaejHi}.
\begin{proposition}
Given $j,r,\ell,s,k\in\NN$, we have
\begin{align} \label{eq:DeltaThetaeHk}
\notag &\Delta_{e_{j-r}} \Theta_{e_{\ell-r+s}} \wH_{(k-s)} =\\ & = \sum_{a=0}^{k-s}q^{\binom{a}{2}}\qbinom{k-s}{a}_q\sum_{b=1}^{j-r+a}\qbinom{b-1}{k-s-1}_qq^{\binom{k-s}{2}-a(k-s-1)}  \Theta_{e_{\ell+k-j-a}}\Delta_{e_{j-r+a}} E_{j-r+a,b} .
\end{align}
\end{proposition}
\begin{proof}
Applying the Delta operator to \eqref{eq:thetaHj} we get
\begin{align*}
\Delta_{e_{j-r}} \Theta_{e_{\ell-r+s}} \wH_{(k-s)} & =\sum_{c=0}^{k-s} (-1)^{k-s-c}q^{c-(k-s)c}q^{\binom{c}{2}}\qbinom{k-s}{c}_q \Delta_{e_{j-r}} e_{\ell+k-r}[X[c]_q]
\end{align*}
Now using \eqref{eq:corDeltaelemXk} we have
\begin{align*}
&\sum_{c=0}^{k-s} (-1)^{k-s-c}q^{c-(k-s)c}q^{\binom{c}{2}}\qbinom{k-s}{c}_q \Delta_{e_{j-r}} e_{\ell+k-r}[X[c]_q]=\\
& = \sum_{c=0}^{k-s} (-1)^{k-s-c}q^{c-(k-s)c}q^{\binom{c}{2}}\qbinom{k-s}{c}_q \sum_{a=0}^{c}q^{\binom{a}{2}}\qbinom{c}{a}_q\times\\
& \times \sum_{b=1}^{j-r+a}\qbinom{c-a+b-1}{c-1}_q   \Theta_{e_{\ell+k-j-a}}\Delta_{e_{j-r+a}} E_{j-r+a,b} \\
& = \sum_{a=0}^{k-s}q^{\binom{a}{2}}\qbinom{k-s}{a}_q\sum_{b=1}^{j-r+a}\sum_{c=a}^{k-s} (-1)^{k-s-c}q^{c-(k-s)c}q^{\binom{c}{2}} \qbinom{k-s-a}{k-s-c}_q\times\\
& \times\qbinom{c-a+b-1}{c-1}_q   \Theta_{e_{\ell+k-j-a}}\Delta_{e_{j-r+a}} E_{j-r+a,b} \\
& = \sum_{a=0}^{k-s}q^{\binom{a}{2}}\qbinom{k-s}{a}_q\sum_{b=1}^{j-r+a}\qbinom{b-1}{k-s-1}_qq^{\binom{k-s}{2}-a(k-s-1)}  \Theta_{e_{\ell+k-j-a}}\Delta_{e_{j-r+a}} E_{j-r+a,b} 
\end{align*}
where in the last equality we used \eqref{eq:lemmaelem}. 
\end{proof}
We can now extend our \eqref{eq:mysummation}. 
\begin{theorem} \label{thm:mygeneralsummation}
Given $j,m,\ell,k\in\NN$, $k\geq 1$ we have
\begin{align} \label{eq:newsummationhjperp}
\notag & h_{j}^\perp\Theta_{e_m}\Theta_{e_{\ell}}\widetilde{H}_{(k)}=\\
& = \sum_{r=0}^{j}\qbinom{k}{r}_q\sum_{a=0}^{k}\sum_{b=1}^{j-r+a} q^{\binom{k-r-a}{2}}\qbinom{b-1}{a}_q  \qbinom{b+r-a-1}{k-a-1}_q \Theta_{e_{m-j+r}}\Theta_{e_{\ell+k-j-a}}\Delta_{e_{j-r+a}}  E_{j-r+a,b} \\
\notag & + \sum_{r=0}^{j}\qbinom{k}{r}_q\sum_{a=0}^{k}\sum_{b=1}^{j-r+a} q^{\binom{k-r-a+1}{2}}\qbinom{b-1}{a-1}_q \qbinom{b+r-a}{k-a}_q  \Theta_{e_{m-j+r}}\Theta_{e_{\ell+k-j-a}}\Delta_{e_{j-r+a}} E_{j-r+a,b} .
\end{align}
\end{theorem}
\begin{proof}
Using twice \eqref{eq:hjek} we have
\begin{align*}
	& h_{j}^\perp\Theta_{e_m}\Theta_{e_{\ell}}\widetilde{H}_{(k)}=\\ 
	& = \sum_{r=0}^{j}\Theta_{e_{m-j+r}} \Delta_{e_{j-r}} h_r^\perp \Theta_{e_{\ell}}\widetilde{H}_{(k)}\\
	& = \sum_{r=0}^{j}\Theta_{e_{m-j+r}} \Delta_{e_{j-r}} \sum_{s=0}^r\Theta_{e_{\ell-r+s}}\Delta_{e_{r-s}} h_s^\perp \widetilde{H}_{(k)}\\
	\text{(using \eqref{eq:hjperpHn} and \eqref{eq:qBmu_Tmu})}	 & = \sum_{r=0}^{j}\Theta_{e_{m-j+r}} \Delta_{e_{j-r}} \sum_{s=0}^r\Theta_{e_{\ell-r+s}}q^{\binom{r-s}{2}}\qbinom{k-s}{r-s}_q \qbinom{k}{s}_q\widetilde{H}_{(k-s)}\\
	& = \sum_{r=0}^{j}\Theta_{e_{m-j+r}}\sum_{s=0}^rq^{\binom{r-s}{2}}\qbinom{k-s}{r-s}_q \qbinom{k}{s}_q  \Delta_{e_{j-r}} \Theta_{e_{\ell-r+s}}\widetilde{H}_{(k-s)}.
\end{align*}
Now using \eqref{eq:DeltaThetaeHk} we finally get
\begin{align*}
& h_{j}^\perp\Theta_{e_m}\Theta_{e_{\ell}}\widetilde{H}_{(k)}=\\ 
& = \sum_{r=0}^{j}\Theta_{e_{m-j+r}}\sum_{s=0}^rq^{\binom{r-s}{2}}\qbinom{k-s}{r-s}_q \qbinom{k}{s}_q  \sum_{a=0}^{k-s}q^{\binom{a}{2}}\qbinom{k-s}{a}_q\times\\
& \times\sum_{b=1}^{j-r+a}\qbinom{b-1}{k-s-1}_qq^{\binom{k-s}{2}-a(k-s-1)}  \Theta_{e_{\ell+k-j-a}}\Delta_{e_{j-r+a}} E_{j-r+a,b} \\ 
& = \sum_{r=0}^{j}\qbinom{k}{r}_q\sum_{s=0}^rq^{\binom{r-s}{2}}\qbinom{r}{r-s}_q   \sum_{a=0}^{k-s}q^{\binom{a}{2}}\qbinom{k-s}{a}_q\sum_{b=1}^{j-r+a}q^{\binom{k-s}{2}-a(k-s-1)}\qbinom{b-1}{k-s-1}_q  \times\\
& \times \Theta_{e_{m-j+r}}\Theta_{e_{\ell+k-j-a}}\Delta_{e_{j-r+a}} E_{j-r+a,b} \\
& = \sum_{r=0}^{j}\qbinom{k}{r}_q\sum_{a=0}^{k}\sum_{b=1}^{j-r+a} q^{\binom{k-r-a}{2}}\qbinom{b-1}{a}_q  \qbinom{b+r-a-1}{k-a-1}_q \Theta_{e_{m-j+r}}\Theta_{e_{\ell+k-j-a}}\Delta_{e_{j-r+a}}  E_{j-r+a,b} \\
& + \sum_{r=0}^{j}\qbinom{k}{r}_q\sum_{a=0}^{k}\sum_{b=1}^{j-r+a} q^{\binom{k-r-a+1}{2}}\qbinom{b-1}{a-1}_q \qbinom{b+r-a}{k-a}_q  \Theta_{e_{m-j+r}}\Theta_{e_{\ell+k-j-a}}\Delta_{e_{j-r+a}} E_{j-r+a,b} ,
\end{align*}
where in the last equality we used \eqref{eq:lemelem2}.
\end{proof}
The case $\ell=m$ of the following corollary is a reformulation of \cite[Theorem~7.6]{dadderio2019theta}.
\begin{corollary}
	\begin{align*}
	& h_{j}^\perp\Delta_{e_{\ell}}e_m[X[k]_q] =  \sum_{a=0}^{k}\sum_{b=1}^{j+a}q^{\binom{k-a}{2}} \qbinom{k}{a}_q  \qbinom{b-1}{k-1}_q t^{j+a-b}\Delta_{h_{j+a-b}}\Delta_{e_{\ell+k-j-a}}e_{m-j}[X[b]_q]  .
	\end{align*}
\end{corollary}
\begin{proof}
	Using \eqref{eq:Thetareciprocity}, \eqref{eq:AperpDeltaomegaA}, \eqref{eq:newsummationhjperp} and Lemma~\ref{lem:hkperpThetamFzero} we have
\begin{align*}
& h_{j}^\perp\Delta_{e_{\ell}}e_m[X[k]_q] = h_{j}^\perp \Delta_{e_{\ell}}h_{k}^\perp \Theta_{e_m} \widetilde{H}_{(k)} = h_{k+\ell}^\perp h_{j}^\perp\Theta_{e_m}\Theta_{e_{\ell}}\widetilde{H}_{(k)}=\\
& = \sum_{r=0}^{j}\qbinom{k}{r}_q\sum_{a=0}^{k}\sum_{b=1}^{j-r+a} q^{\binom{k-r-a}{2}}\qbinom{b-1}{a}_q  \qbinom{b+r-a-1}{k-a-1}_q h_{k+\ell}^\perp\Theta_{e_{m-j+r}}\Theta_{e_{\ell+k-j-a}}\Delta_{e_{j-r+a}}  E_{j-r+a,b} \\
& + \sum_{r=0}^{j}\qbinom{k}{r}_q\sum_{a=0}^{k}\sum_{b=1}^{j-r+a} q^{\binom{k-r-a+1}{2}}\qbinom{b-1}{a-1}_q \qbinom{b+r-a}{k-a}_q  h_{k+\ell}^\perp\Theta_{e_{m-j+r}}\Theta_{e_{\ell+k-j-a}}\Delta_{e_{j-r+a}} E_{j-r+a,b}\\
& = \sum_{a=0}^{k}\sum_{b=1}^{j+a} q^{\binom{k-a}{2}}\qbinom{b-1}{a}_q  \qbinom{b-a-1}{k-a-1}_q h_{k+\ell}^\perp\Theta_{e_{m-j}}\Theta_{e_{\ell+k-j-a}} t^{j+a-b}\Theta_{h_{j+a-b}}  \wH_{(b)} \\
& + \sum_{a=0}^{k}\sum_{b=1}^{j+a} q^{\binom{k-a+1}{2}}\qbinom{b-1}{a-1}_q \qbinom{b-a}{k-a}_q  h_{k+\ell}^\perp\Theta_{e_{m-j}}\Theta_{e_{\ell+k-j-a}}t^{j+a-b}\Theta_{h_{j+a-b}}  \wH_{(b)} \\
& = \sum_{a=0}^{k}\sum_{b=1}^{j+a} q^{\binom{k-a}{2}}\qbinom{b-1}{a}_q  \qbinom{b-a-1}{k-a-1}_q t^{j+a-b}\Delta_{h_{j+a-b}}\Delta_{e_{\ell+k-j-a}}e_{m-j}[X[b]_q] \\
& + \sum_{a=0}^{k}\sum_{b=1}^{j+a} q^{\binom{k-a+1}{2}}\qbinom{b-1}{a-1}_q \qbinom{b-a}{k-a}_q  t^{j+a-b}\Delta_{h_{j+a-b}}\Delta_{e_{\ell+k-j-a}}e_{m-j}[X[b]_q]\\
& = \sum_{a=0}^{k}\sum_{b=1}^{j+a}q^{\binom{k-a}{2}} \qbinom{k}{a}_q  \qbinom{b-1}{k-1}_q t^{j+a-b}\Delta_{h_{j+a-b}}\Delta_{e_{\ell+k-j-a}}e_{m-j}[X[b]_q]  .
\end{align*}
where in the last equality we used \eqref{eq:lemelem3}.
\end{proof}

\section{Even more new identities}

We start with the following new identity.

\begin{theorem}
Given $j,m,p,k$ with $k\geq 1$ we have
\begin{equation} \label{eq:thmimplygenDelta}
h_{j}^\perp t^p\Theta_{h_p}\Theta_{e_m} \widetilde{H}_{(k)} =\sum_{s=0}^{j}t^{j-s}\sum_{r=0}^sq^{\binom{s-r}{2}}\qbinom{k-r}{s-r}_q\qbinom{k}{r}_q\Delta_{h_{j-s}}t^{p-j+s}\Theta_{h_{p-j+s}}\Theta_{e_{m-s+r}}\widetilde{H}_{(k-r)}.
\end{equation}
\end{theorem}
\begin{proof}
Using \eqref{eq:hjhk} we have
\begin{align*}
h_{j}^\perp t^p\Theta_{h_p}\Theta_{e_m} \widetilde{H}_{(k)}& =t^p\sum_{s=0}^{j}\Delta_{h_{j-s}}\Theta_{h_{p-j+s}}h_s^\perp \Theta_{e_m} \widetilde{H}_{(k)}\\ 
\text{(using \eqref{eq:hjek})}	&=t^p\sum_{s=0}^{j}\Delta_{h_{j-s}}\Theta_{h_{p-j+s}}\sum_{r=0}^s\Theta_{e_{m-s+r}}\Delta_{e_{s-r}}h_r^\perp\widetilde{H}_{(k)}\\ 
\text{(using \eqref{eq:hjperpHn})}	&=\sum_{s=0}^{j}t^{j-s}\sum_{r=0}^sq^{\binom{s-r}{2}}\qbinom{k-r}{s-r}_q\qbinom{k}{r}_q\Delta_{h_{j-s}}t^{p-j+s}\Theta_{h_{p-j+s}}\Theta_{e_{m-s+r}}\widetilde{H}_{(k-r)}.
\end{align*}
\end{proof}

The following result can be used to deduce a combinatorial interpretation of $\Delta_{h_{\ell}}\Delta_{e_n} E_{n,k}$ from the one of $\Delta_{e_n} E_{n,k}$ given by the shuffle theorem \cite{Carlsson-Mellit-ShuffleConj-2015}. For example for $m=0$ it can be interpreted as ``pushing in'' the $j$ big cars into zeros, when the big cars are not on the diagonal, and removing them if they are on the diagonal (see \cite[Theorem~7.5]{dadderio2019theta}). This would provide a simplification of the argument used in \cite{dadderio2019theta} to prove the \emph{generalized shuffle conjecture}. For details, we refer to a forthcoming article\footnote{Private communication.} in which Iraci and Vanden Wyngaerd show how in fact the following identity can be used to prove that the generalized valley Delta conjecture of Qiu and Wilson \cite{Qiu_Wilson} (and hence its square version in \cite{IraciVandenWyngaerd_valleyDeltaSquare}) is implied by the valley Delta conjecture \cite{Haglund-Remmel-Wilson-2015} (which is still open).
\begin{corollary} \label{cor:genDeltaId}
	Given $j,m,p,k$ with $k\geq 1$ we have
\begin{equation}
h_j^\perp \Theta_{e_m} \Delta_{e_{p+k}} E_{p+k,k} = \sum_{s=0}^{j}t^{j-s}\sum_{r=0}^sq^{\binom{s-r}{2}}\qbinom{k-r}{s-r}_q\qbinom{k}{r}_q \Delta_{h_{j-s}} \Theta_{e_{m-s+r}} \Delta_{e_{p-j+s+k-r}} E_{p-j+s+k-r,k-r} .
\end{equation}	
\end{corollary}
\begin{proof}
Combine \eqref{eq:thmimplygenDelta} and \eqref{eq:nablaEnk}.
\end{proof}
The following theorem seems new.
\begin{theorem}
Given $m,k\in\NN$, $k\geq 1$,
\begin{equation} \label{eq:Deltat0}
\Theta_{e_m}\wH_{(k)}=\left.\Delta_{e_{k-1}}'e_{m+k}\right|_{t=0}.
\end{equation}
\end{theorem}
\begin{proof}
The touching refinement of the Delta conjecture (see \cite[Conjecture~5.2]{dadderio2019theta}, also for the missing definitions), now proved in \cite{DAdderio_Mellit2020}, states that $\Theta_{e_m}\Delta_{e_k} E_{k,r}$ is the sum of certain weights of labelled Dyck paths of size $k+m$ with $m$ decorated rises touching the main diagonal $y=x$ at precisely $r$ points (ignoring $(k+m,k+m)$). The weights are monic monomials in the variables $q,t,x_1,x_2,\dots$, whose power of $t$ is given by the area of the decorated Dyck path. It is obvious from the definitions that if $r<k$ then the area of such a path must be positive. Therefore
\[ \left.\Theta_{e_m}\Delta_{e_k} E_{k,r}\right|_{t=0}=0\quad \text{if }r<k. \]

Now combining this with \eqref{eq:Enkformula} and \eqref{eq:Deltaprimeen}, and using \eqref{eq:nablaEnk} we get \[\left.\Delta_{e_{k-1}}'e_{m+k}\right|_{t=0}=\sum_{r=1}^{k}\left.\Theta_{e_m}\Delta_{e_k} E_{k,r}\right|_{t=0}=\left.\Theta_{e_m}\wH_{(k)}\right|_{t=0}=\Theta_{e_m}\wH_{(k)},\]
where the last equality follows from \eqref{eq:thetaHj}.
\end{proof}
As a corollary, we get the following identity, which is a reformulation of \cite[Proposition~9.2]{DAdderio-Iraci-VandenWyngaerd-Delta-Square}.
\begin{corollary}
	Given $j,m,k$ with $k\geq 1$ we have
	\begin{equation} 
	h_{j}^\perp \left.\Delta_{e_{k-1}}'e_{m+k}\right|_{t=0}  =\sum_{r=0}^jq^{\binom{j-r}{2}}\qbinom{k-r}{j-r}_q\qbinom{k}{r}_q \left.\Delta_{e_{k-r-1}}'e_{m-j+k}\right|_{t=0} .
	\end{equation}
\end{corollary}
\begin{proof}
Just combine the special case $p=0$ of \eqref{eq:thmimplygenDelta} with \eqref{eq:Deltat0}.
\end{proof}

\subsection{Further applications}
The following result first appeared in \cite[Lemma~3.7]{Haglund-Rhoades-Shimozono-Advances}.
\begin{corollary}
	Given $m,k,j\in\NN$, we have
\begin{align}
\notag & e_j^\perp q^{\binom{m+k}{2}-\binom{m+1}{2}}\overline{\omega}\left.\Delta_{e_{k-1}}'e_{m+k}\right|_{t=0}=\\
& = \sum_{r=0}^jq^{\binom{j}{2}+r(m-j+r)}\qbinom{k-r}{j-r}_q\qbinom{k}{r}_q   q^{\binom{m+k-j}{2}-\binom{m-j+r+1}{2}}\overline{\omega}\left.\Delta_{e_{k-r-1}}'e_{m+k-j}\right|_{t=0}.
\end{align}
\end{corollary}
\begin{proof}
Just apply $\overline{\omega}$ to \eqref{eq:thmimplygenDelta} with $p=0$, and then use \eqref{eq:Deltat0}.
\end{proof}
The following result first appered in \cite[Lemma~3.1]{Haglund_Rhoades_Shimozono_SchurDelta}.
\begin{corollary}
	Given $m,k,j\in\NN$, we have
\begin{equation}
e_j^\perp \left.\Delta_{e_{k-1}}'e_{m+k}\right|_{t=0} = \sum_{r=0}^j q^{\binom{r}{2}}\qbinom{k}{r}_q\qbinom{k+j-r-1}{j-r}_q\left.\Delta_{e_{k-r-1}}'e_{m+k-j}\right|_{t=0}.
\end{equation}
\end{corollary}
\begin{proof}
	Using \eqref{eq:Deltat0} we have
\begin{align*}
e_j^\perp \left.\Delta_{e_{k-1}}'e_{m+k}\right|_{t=0} & = e_j^\perp \Theta_{e_m}\wH_{(k)}\\
\text{(using \eqref{eq:ejek})} & = \sum_{r=0}^j \Theta_{e_{m-j+r}} e_r^\perp \Delta_{h_{j-r}}\wH_{(k)}\\
\text{(using \eqref{eq:hjperpHn})}& = \sum_{r=0}^j q^{\binom{r}{2}}\qbinom{k}{r}_q\qbinom{k+j-r-1}{j-r}_q \Theta_{e_{m-j+r}} \wH_{(k-r)}\\
\text{(using \eqref{eq:Deltat0})}& = \sum_{r=0}^j q^{\binom{r}{2}}\qbinom{k}{r}_q\qbinom{k+j-r-1}{j-r}_q \left.\Delta_{e_{k-r-1}}'e_{m+k-j}\right|_{t=0}.
\end{align*}
\end{proof}

The following result first appeared in \cite[Theorem~3.4]{DAdderio-Iraci-VandenWyngaerd-GenDeltaSchroeder}, though the special case $j=0$ already appeared in \cite[Theorem~4.17]{DAdderio-Iraci-VandenWyngaerd-Bible}.
\begin{theorem}
Given $m,k,\ell,j,p\in\NN$, $m\geq 1$ we have
\begin{align}
 & t^p\<\Delta_{h_p}\Delta_{e_{\ell}}e_m[X[k]_q],e_jh_{m-j}\> =\\
 \notag &=t^p \sum_{r=0}^{k}q^{\binom{r}{2}}\qbinom{k}{r}_q\sum_{b=1}^{j+r}\qbinom{k-r+b-1}{k-1}_q t^{j+r-b}\<\Delta_{h_{j+r-b}}\Delta_{e_{m-j-r}}e_{p+\ell}[X[b]_q],e_p h_\ell\> .
\end{align}
\end{theorem}
\begin{proof}
We have 
\begin{align*}
& t^p\<\Delta_{h_p}\Delta_{e_{\ell}}e_m[X[k]_q],e_jh_{m-j}\> = \\
\text{(using \eqref{eq:hkstarenmkstar})}& = t^p\<\Delta_{e_j}\Delta_{h_p}\Delta_{e_{\ell}}e_m[X[k]_q],h_{m}\>  \\
& =t^ph_m^\perp \Delta_{h_p}\Delta_{e_{\ell}}\Delta_{e_j}e_m[X[k]_q] \\
\text{(using \eqref{eq:AperpDeltaomegaA})}& = t^p e_p^\perp h_\ell^\perp h_m^\perp  \Theta_{e_{p+\ell}}\Delta_{e_j}e_m[X[k]_q]\\
\text{(using \eqref{eq:DeltaFem})}& =t^p e_p^\perp h_\ell^\perp h_m^\perp\Theta_{e_{p+\ell}} h_{k+j}^\perp \Theta_{e_m} \Theta_{e_j} \widetilde{H}_{(k)}\\
\text{(using \eqref{eq:mysummation})}& = t^p e_p^\perp h_\ell^\perp h_m^\perp \Theta_{e_{p+\ell}}  \sum_{r=0}^{k}q^{\binom{r}{2}}\qbinom{k}{r}_q\sum_{b=1}^{j+r}\qbinom{k-r+b-1}{k-1}_q  t^{j+r-b}\Theta_{h_{j+r-b}}\Theta_{e_{m-j-r}}\widetilde{H}_{(b)}\\
\text{(using \eqref{eq:DeltaFem})}& = t^p \sum_{r=0}^{k}q^{\binom{r}{2}}\qbinom{k}{r}_q\sum_{b=1}^{j+r}\qbinom{k-r+b-1}{k-1}_q t^{j+r-b}\<\Delta_{h_{j+r-b}}\Delta_{e_{m-j-r}}e_{p+\ell}[X[b]_q],e_p h_\ell\> .
\end{align*}
\end{proof}

\section{Appendix: proofs of elementary lemmas}

In this appendix we prove some of the elementary lemmas that we used in the text.

\subsection{Proof of \eqref{eq:lemmaelem}}
From the summation formula \eqref{eq:SchurSummation} for homogeneous symmetric functions, we get 
\begin{align*}
h_{i-1}\left[  q^{i-a} [b-i+1]_q  \right] 
&  =    h_{i-1}\left[ [b-a+1]_q  - [i-a]_q \right]  \\
&  =     \sum_{c=1}^i h_{c-1}\left[ [ b-a+1]_q \right] h_{i-1-(c-1)}\left[  -[i-a ]_q \right]  \\ 
&  =     \sum_{c=1}^i h_{c-1}\left[ [ b-a+1]_q \right] e_{i-c}\left[ [i-a   ]_q \right]  (-1)^{i-c}. \\ 
\end{align*}
 Now using the evaluations \eqref{eq:h_q_binomial} and \eqref{eq:e_q_binomial}, we have 
 \begin{align*}
q^{(i-a)(i-1)} { b-1 \brack i-1 }_q
 =    \sum_{c=1}^i { c-a+b-1  \brack  c-1 }_q   { i-a  \brack i-c   }_q  q^{\binom{i-c}{2}}  (-1)^{i-c} 
\end{align*}
The sum can start at $c=a$ since the second binomial would otherwise be $0$, and multiplying both sides by $q^{-\binom{i}{2}}$ we get the lemma.

\begin{remark} \label{rem:elLemma}
Notice that \eqref{eq:first_qlemma} can be proved with a similar argument, starting with $h_{i+a}[q^{i-1}[s-i+1]_q]=h_{i+a}[[s]_q-[i-1]_q]$.
\end{remark}

\subsection{Proof of \eqref{eq:lemelem2}}

	Using \eqref{eq:qrecurrence} we have
	\begin{align*}
	& \sum_{s=0}^rq^{\binom{r-s}{2}}\qbinom{r}{r-s}_q   q^{\binom{a}{2}}\qbinom{k-s}{a}_qq^{\binom{k-s}{2}-a(k-s-1)}\qbinom{b-1}{k-s-1}_q =\\
	& =\sum_{s=0}^rq^{\binom{r-s}{2}+\binom{a}{2}+\binom{k-s}{2}-a(k-s-1)}\qbinom{r}{s}_q   \qbinom{k-s-1}{a}_q\qbinom{b-1}{k-s-1}_q\\
	& +\sum_{s=0}^rq^{\binom{r-s}{2}+\binom{a}{2}+\binom{k-s}{2}-a(k-s-1)}\qbinom{r}{s}_q   q^{k-s-a}\qbinom{k-s-1}{a-1}_q\qbinom{b-1}{k-s-1}_q\\
	& =\qbinom{b-1}{a}_q\sum_{s=0}^rq^{\binom{r-s}{2}+\binom{a}{2}+\binom{k-s}{2}-a(k-s-1)}\qbinom{r}{s}_q   \qbinom{b-a-1}{k-s-a-1}_q\\
	& +\qbinom{b-1}{a-1}_q\sum_{s=0}^rq^{\binom{r-s}{2}+\binom{a}{2}+\binom{k-s}{2}-a(k-s-1)+(k-s-a)}\qbinom{r}{s}_q   \qbinom{b-a}{k-s-a}_q.
	\end{align*}
	Now using the elementary
	\begin{align*}
	& \binom{r-s}{2}+\binom{a}{2}+\binom{k-s}{2}-a(k-s-1) =\\
	& = \frac{1}{2}(r^2-2rs+s^2-r+s +a^2-a +k^2+s^2-2ks-k+s-2ak+2as+2a)\\
	& = s^2-ks-rs+as+s-r-ar+kr +\frac{1}{2}(r^2+r+2ar +a^2+a-2kr +k^2-k-2ak)\\
	& = (r-s)(k-s-a-1) +\binom{k-r-a}{2},
	\end{align*}
	so that
	\begin{align*}
	& \binom{r-s}{2}+\binom{a}{2}+\binom{k-s}{2}-a(k-s-1)+(k-s-a) =\\
	& = (r-s)(k-s-a-1)+(k-s-a) +\binom{k-r-a}{2}\\
	& = (r-s)(k-s-a)+(k-r-a) +\binom{k-r-a}{2}\\
	& = (r-s)(k-s-a)+\binom{k-r-a+1}{2},
	\end{align*}
	we get
	\begin{align*}
	& \qbinom{b-1}{a}_q\sum_{s=0}^rq^{\binom{r-s}{2}+\binom{a}{2}+\binom{k-s}{2}-a(k-s-1)}\qbinom{r}{s}_q   \qbinom{b-a-1}{k-s-a-1}_q+\\
	& +\qbinom{b-1}{a-1}_q\sum_{s=0}^rq^{\binom{r-s}{2}+\binom{a}{2}+\binom{k-s}{2}-(a-1)(k-s)}\qbinom{r}{s}_q   \qbinom{b-a}{k-s-a}_q=\\
	& =q^{\binom{k-r-a}{2}}\qbinom{b-1}{a}_q\sum_{s=0}^rq^{(r-s)(k-s-a-1)}\qbinom{r}{s}_q   \qbinom{b-a-1}{k-s-a-1}_q+\\
	& +q^{\binom{k-r-a+1}{2}}\qbinom{b-1}{a-1}_q\sum_{s=0}^rq^{ (r-s)(k-s-a)}\qbinom{r}{s}_q   \qbinom{b-a}{k-s-a}_q\\
	& =q^{\binom{k-r-a}{2}}\qbinom{b-1}{a}_q  \qbinom{b+r-a-1}{k-a-1}_q+q^{\binom{k-r-a+1}{2}}\qbinom{b-1}{a-1}_q \qbinom{b+r-a}{k-a}_q,
	\end{align*}
	where in the last equality we used \eqref{eq:qChuVander}.	
\subsection{Proof of \eqref{eq:lemelem3}}

We have
	\begin{align*}
	& q^{\binom{k-a}{2}}\qbinom{b-1}{a}_q  \qbinom{b-a-1}{k-a-1}_q + q^{\binom{k-a+1}{2}}\qbinom{b-1}{a-1}_q \qbinom{b-a}{k-a}_q =\\
	& = q^{\binom{k-a}{2}}\qbinom{b-1}{a}_q  \qbinom{b-a-1}{k-a}_q \left(\frac{[k-a]_q}{[b-k]_q}+q^{k-a}\frac{[a]_q}{[b-a]_q}\frac{[b-a]_q}{[b-k]_q}\right) \\
	& = q^{\binom{k-a}{2}}\frac{[k-1]_q!}{[k-1]_q!}\frac{[b-1]_q!}{[a]_q![k-a]_q![b-k-1]_q!} \frac{[k]_q}{[b-k]_q}\\
	& = q^{\binom{k-a}{2}} \qbinom{k}{a}_q  \qbinom{b-1}{k-1}_q.
	\end{align*}

\bibliographystyle{amsalpha}
\bibliography{Biblebib}

\end{document}